\documentclass[11pt,reqno]{amsart}
\usepackage[latin1]{inputenc}
\usepackage{amssymb}
\usepackage{pdfsync}
\usepackage[english]{babel}

\usepackage{tikz-cd}
\usepackage{adjustbox}

\usepackage[pdftex, colorlinks=true,urlcolor=black,linkcolor=blue,citecolor=blue, linktocpage]{hyperref}
\usepackage{graphicx}

\usepackage{fullpage}
\usepackage{color}
\usepackage{amsmath}
\usepackage{amsfonts}
\usepackage{amsthm}
\usepackage{mathrsfs}
\usepackage{t1enc, graphicx}
\usepackage{verbatim}
\usepackage{bbm}
\usepackage{enumitem}
\usepackage{centernot}
\usepackage{mathtools}
\usepackage[capitalize]{cleveref}

\usepackage[left= 1.2 in, right= 1.2 in ,top= 1.2 in, bottom = 2.1 in]{geometry}

\newcommand{\R}{\mathbb{R}}

\newcommand{\C}{\mathbb{C}}
\newcommand{\T}{\mathbb{T}}
\newcommand{\Q}{\mathbb{Q}}
\newcommand{\Z}{\mathbb{Z}}
\newcommand{\N}{\mathbb{N}}

\newcommand{\E}{\mathbb{E}}
\renewcommand{\P}{\mathbb{P}}

\newcommand{\Bohr}{\operatorname{Bohr}}

\newcommand{\spec}{\operatorname{spec}}

\newcommand{\DiffBohr}{\operatorname{(\mathcal{D}, \mathcal{B})}\text{-expanding}}
\newcommand{\AlmostBohrBohr}{\operatorname{(\mathcal{A}, \mathcal{B})}\text{-expanding}}
\newcommand{\DiffG}{\operatorname{(\mathcal{D}, \{G\})}\text{-expanding}}
\newcommand{\KMF}{\operatorname{KMF}}

\newtheorem{theorem}{Theorem}[section]
\newtheorem*{theorem*}{Theorem}
\newtheorem{proposition}[theorem]{Proposition}
\newtheorem{lemma}[theorem]{Lemma}

\newtheorem{corollary}[theorem]{Corollary}
\newtheorem*{corollary*}{Corollary}
\newtheorem{question}[theorem]{Question}
\newtheorem{conjecture}[theorem]{Conjecture}

\theoremstyle{definition}
\newtheorem*{definition*}{Definition}
\newtheorem{definition}[theorem]{Definition}

\theoremstyle{remark}

\newtheorem*{remark*}{Remark}
\newtheorem{remark}[theorem]{Remark}

\theoremstyle{plain}
\newcounter{MainTheoremCounter}

\newtheorem{Maintheorem}[MainTheoremCounter]{Theorem}

\theoremstyle{plain}
\newcounter{OldTheoremCounter}

\newcommand{\vertiii}[1]{{\left\vert\kern-0.25ex\left\vert\kern-0.25ex\left\vert #1 
    \right\vert\kern-0.25ex\right\vert\kern-0.25ex\right\vert}}

\crefformat{MainTheoremCounter}{Theorem #2#1#3}

\author{Pierre-Yves Bienvenu}

\author{John T. Griesmer}    
\author{Anh N. Le}
\author{Th\'ai Ho\`ang L\^e}

\address{Johann Radon Institute for Computational and Applied Mathematics, Austrian Academy of Sciences, Linz, Austria} 
\email{pierre.bienvenu@oeaw.ac.at}

\address{Department of Applied Mathematics and Statistics\\
	Colorado School of Mines\\
	1005 14th Street, Golden, CO 80401, USA} 
\email{jtgriesmer@gmail.com}

\address{Department of Mathematics\\
	University of Denver\\
	2390 S. York St, Denver, CO 80210, USA}
\email{anh.n.le@du.edu}

\address{Department of Mathematics\\
	University of Mississippi\\
	University, MS 38677, USA}
\email{leth@olemiss.edu}

\title{Bohr sets in sumsets III: Expanding Difference sets and almost Bohr sets}

\begin{document}
\begin{abstract}

Let $G$ be a discrete abelian group. F{\o}lner showed that if $A \subseteq G$ has positive upper Banach density, then $A - A$ contains an \emph{almost Bohr set} -- a set of the form $B \setminus E$ where $B$ is a Bohr set and $E$ has zero Banach density.

We study the sets $S \subseteq G$ for which $A - A + S$ contains a Bohr set for every $A \subseteq G$ of positive upper Banach density. For $G = \Z$, we show that the sets $\{n^2: n \in \N\}$, $\{p - 1: p \text{ prime}\}$, and $\{ \lfloor n^c \rfloor: n \in \N \}$ with $c > 0$, have this property. Moreover, we prove that there are sets $A, B \subseteq \Z$ such that $A$ is dense in the Bohr topology of $\Z$, $d^*(B) > 0$, while $A + B$ is not piecewise Bohr, answering two questions of the second author in \cite{griesmer-dense-set}.


We also study those sets $S$ such that $A + S$ contains a Bohr set for every almost Bohr set $A$. As applications, we prove:
  \begin{enumerate}
    \item If $\phi_1, \phi_2: G \to G$ are (not necessarily commuting) homomorphisms with finite indices $[G: \phi_i(G)]$, and $C \subseteq G$ is a central set, then $\phi_1(C) - \phi_1(C) + \phi_2(C)$ contains a Bohr set. This answers one of our questions in \cite{Griesmer-Le-Le-countable} and generalizes results in \cite{Le-Le-compact,Le-Wathodkar};

    \item Every set of pointwise recurrence in $\Z$ is a set of nice recurrence and a van der Corput set, extending known properties of sets of pointwise recurrence studied in \cite{Glasscock-Le-syndetic, glasscock_le_pointwise, HostKraMaass2016}.
  \end{enumerate}
\end{abstract}

\maketitle

\tableofcontents

\section{Introduction}

\subsection{History and motivation}
\label{sec:history}

This paper follows \cite{Griesmer-Le-Le-countable} and \cite{Le-Le-compact} in the study of Bohr structure in sumsets. 
Here we will deal with an arbitrary discrete abelian group $(G, +)$ (countable or uncountable). If $A, B \subseteq G$, the \emph{sumset} and \emph{difference set} of $A$ and $B$ are $A + B: = \{a + b: a \in A, b \in B \}$ and $A - B:=\{a - b: a \in A, b \in B \}$, respectively.  For $a\in G$, the \emph{translate} $a+B$ is $\{a+B:b\in B\}$. If $s \in \Z$, we define $sa := a + a + \cdots + a$, the sum of $s$ copies of $a$, if $s \geq 0$ and $s a = - (-s)a$ if $s < 0$. Likewise, $sA:=\{sa: a \in A\}$ for $A \subseteq G$.  A \emph{character} of $G$ is a homomorphism from $G$ to $S^1:=\{z\in \mathbb C: |z|=1\}$. Throughout this paper, $\Z$ denotes the set of integers and $\N$ denotes the set of positive integers $\{1, 2, 3, \ldots\}$.

Many classical results in additive combinatorics state roughly that sumsets are more structured than their summands.  Such results often quantify the structure found in sumsets in terms of Bohr sets, which we now define.
For a finite set $\Lambda$ of characters of $G$  and a constant $\eta > 0$, the set 
\begin{equation*} 
    B (\Lambda; \eta):= \{ x \in G : | \chi(x)-1 | < \eta \textup{ for all } \chi \in \Lambda\}
\end{equation*} 
is called a \emph{Bohr set}, or a \emph{Bohr neighborhood of $0$}.  The set $B(\Lambda; \eta)$ is also called a \emph{Bohr-$(k, \eta)$} set where $k = |\Lambda|$. We refer to $\eta$ as the \emph{radius} and $k$ as the \emph{rank} of the Bohr set. We call $\Lambda$ the set of \emph{frequencies} determining $B(\Lambda;\eta)$.  For $g\in G$, a \emph{Bohr neighborhood of $g$} is a set of the form $g+B$, where $B$ is a Bohr set.



The study of Bohr sets in sumsets started with the following  theorem of Bogolyubov \cite{Bogolyubov}; see \cref{sec:def_notation} for the definition of upper Banach density $d^*(A)$.

\begin{theorem}[\cite{Bogolyubov, Folner2}] \label{th:bogo}
If $A \subseteq \mathbb Z$ satisfies $d^*(A) > 0$, 
then $A - A + A - A$ contains a Bohr set whose rank and radius depend only on $d^*(A)$.
\end{theorem}


F{\o}lner \cite{Folner1} extended Bogolyobov's theorem from $\mathbb Z$ to topological abelian groups, including discrete abelian groups of arbitrary cardinality.
While Bogolyubov's intended application for Theorem \ref{th:bogo} was the study of almost periodic functions, it is now a standard tool in additive combinatorics. It was used in Ruzsa's proof of Freiman's theorem \cite{Ruzsa_Generalized} and in Gowers's proof of Szemer\'edi's theorem \cite{Gowers01}.  

 In \cite{Folner1}, F{\o}lner  showed that the last two summands in Bogolyubov's theorem are ``almost'' redundant by proving that $A - A$ already contains most of a Bohr set. 

\begin{theorem}[\cite{Folner1, Folner2}]
\label{thm:folner}
 If $A \subseteq G$ satisfies $d^*(A) > 0$, then $A - A$ contains a set of the form $B \setminus E$ where $B$ is a Bohr set and $d^*(E) = 0$.
\end{theorem}
The exceptional set $E$ is unavoidable: Kriz  \cite{Kriz87} demonstrated that there exists a set $A\subseteq \mathbb Z$ having positive upper Banach density for which $A - A$ contains no Bohr set; see \cite{Ruzsa85} for an alternative presentation and details of how Kriz's construction relates to Bohr sets. The second author \cite{Griesmer_SeparatingBohr} showed that there is a set $A\subseteq \mathbb Z$ having $d^*(A)>0$ such that $A-A$ contains no Bohr neighborhood of any integer.  


While F{\o}lner and Bogolyubov's results involve two and four summands respectively, Bohr structure with three summands was first studied in $\mathbb Z$ by Bergelson and Ruzsa \cite{Bergelson-Ruzsa}. They show that if $s_1, s_2, s_3$ are non-zero integers satisfying $s_1 +s_2+s_3 = 0$ and $A \subseteq \Z$ has positive upper Banach density, then $s_1 A+s_2 A+s_3 A$ contains a Bohr set. They also show that the hypothesis $s_1+s_2+s_3=0$ cannot be omitted.

The latter theorem of Bergelson and Ruzsa in turn was generalized by the second author in \cite{Griesmer-threefold-sumset} and further generalized by the last three authors to the setting of discrete countable abelian groups:

\begin{theorem}[{\cite[Theorem 1.2]{Griesmer-Le-Le-countable}}] \label{th:main-density}
Let $G$ be a discrete countable abelian group. Let $\phi_1, \phi_2, \phi_3 : G \rightarrow G$ be commuting homomorphisms such that $\phi_1 + \phi_2 +\phi_3 =0$ and $[G:\phi_j(G)]$ are finite for $j \in \{1, 2, 3\}$.
Suppose $A \subseteq G$ satisfies $d^*(A) >0$. Then the set
\[
\phi_1(A) + \phi_2(A) + \phi_3(A)
\]
contains a Bohr set whose rank and radius depend only on $d^*(A)$ and $[G:\phi_j(G)]$.
\end{theorem}

The hypothesis that $[G:\phi_j(G)]$ is finite cannot be omitted: taking $A\subseteq \mathbb Z$ to be a set with $d^*(A)>0$ and $A-A$ does not contain a Bohr set (as given in \cite{Kriz87} or \cite{Griesmer_SeparatingBohr}), $\phi_1=-\phi_2=$ the identity map on $\mathbb Z$ and $\phi_3=0$, we get $A-A=\phi_1(A)+\phi_2(A)+\phi_3(A)$, so the latter set does not contain a Bohr set.

\subsection{Expanding difference sets}


With the above results in mind, we consider the following definition and question.

\begin{definition}
Let $G$ be a discrete abelian group.  A set $S \subseteq G$ is called a \emph{$\DiffBohr$} set if for every $A \subseteq G$ with $d^*(A) > 0$, the sumset $A - A + S$ contains a Bohr set. 
\end{definition}
In the definition above, $\mathcal{D}$ stands for ``difference sets'', and $\mathcal{B}$ for ``Bohr sets''.

 \begin{question}\label[question]{ques:complete_difference}
  Which sets $S \subseteq G$ are $\DiffBohr$?
 \end{question} 

Theorems \ref{th:SpectralSumDifference_intro}, \ref{cor:Specific-intro_main},  \ref{mainthm:difference_set_not_DF}, and \ref{th:SpectralSumDifference}  below address this question.  We will see that the class of $\DiffBohr$ sets occupies a new niche
in the hierarchy of recurrence properties, which includes sets of recurrence,  van der Corput sets, and others considered in dynamical approaches to additive combinatorics.

Our first result provides a sufficient condition for a set to be $\DiffBohr$.   The definition of \emph{$\KMF$-mean}, which we postpone to \cref{sec:def_notation}, isolates the properties used in Kamae and Mend{\`e}s France's generalization  \cite{Kamae_France78} of the Furstenberg-S{\'a}rk{\"o}zy theorem \cite{Furstenberg77, Sarkozy78}.


\begin{Maintheorem}\label[Maintheorem]{th:SpectralSumDifference_intro}  
If $S \subseteq G$ supports a $\KMF$ mean, then $S$ is a $\DiffBohr$ set.
\end{Maintheorem}


Using \cref{th:SpectralSumDifference_intro}, we can exhibit many examples of $\DiffBohr$ sets in $\Z$. For example, the sets $\{n^2: n \in \N\}$, $\{p-1: p \text{ prime}\}$, and $\{\lfloor n^c \rfloor: n \in \N\}$ are such sets, where $c > 0$ and $\lfloor \cdot \rfloor$ denotes the integer part. Moreover, a strengthening of Theorem \ref{th:SpectralSumDifference_intro} (see Theorem \ref{th:SpectralSumDifference} below) uses the \emph{spectrum} (see \cref{sec:kmf_mean_def}) of a KMF mean supported on $S$ to further specify the Bohr set contained in $A-A+S$. For example, the Bohr set could be a subgroup of finite index, or even the whole group.

The polynomial $P(n) = n^2$ is a classical example of an intersective polynomial; we say $P(x) \in \Z[x]$ is \textit{intersective} if for any positive integer $m$, there exists $r \in \Z$ such that $P(r)$ is divisible by $m$. Following \cite{LeThaiHoang_survey}, if one can moreover choose such an $r$ with $\gcd(m,r)=1$, then we say that $P$ is \textit{intersective of the second kind}. 

It is easy to see that every polynomial with at least one integer root is intersective. There are intersective polynomials having no integer roots, such as $P(n) = (n^2 - 13)(n^2 - 17)(n^2 - 221)$; cf.~\cite{Berend_Bilu96}. If the polynomial $P$ satisfies $P(0) = 0$, then the shifted polynomials $Q(n) = P(n+1)$ and $Q(n) = P(n-1)$ are intersective polynomials of the second kind. The following theorem showcases the applicability of Theorems \ref{th:SpectralSumDifference_intro} and \ref{th:SpectralSumDifference}.


\begin{Maintheorem}\label{cor:Specific-intro_main}
     Let $A \subseteq \Z$ with $d^*(A) > 0$. 
    \begin{enumerate}
        \item
        Let $P \in \Z[x]$ be a nonconstant intersective polynomial and $S:=\{P(n):n\in \mathbb N\}$. Then $A - A + S$ contains a finite index subgroup of $\Z$.

       \item
       Let $P \in \Z[x]$ be a nonconstant intersective polyomial of the second kind and let $S := \{P(p): p \text{ prime}\}$. Then $A - A + S$ contains a finite index subgroup of $\Z$.

        \item 
        If $S := \{\lfloor n^c \rfloor: n \in \N\}$ where $c > 0, c \not \in \Z$, then $A - A + S = \Z$.

        \item 
        Moreover, if $S$ is any of the above sets and $S = S_1 \cup S_2 \cup \cdots \cup S_k$, then one of the $S_i$ is $\DiffBohr$.
    \end{enumerate}
   
\end{Maintheorem}
We derive Theorem \ref{cor:Specific-intro_main} from Theorems \ref{th:improved_bulinski_fish_polynomial}, \ref{th:improved_bulinski_fish_polynomial_prime}, and \ref{thm:bulinski_fish_hardy_field}. These are stronger results whose hypotheses allow further restrictions  on $n$ and $p$ in parts (i), (ii) and (iii) of Theorem \ref{cor:Specific-intro_main}.  We draw inspiration for pursuing these stronger results from the work of Bulinski and Fish in \cite{BulinskiFish_QuantitativeTwisted}.

Since $0$ belongs to every Bohr set, every $\DiffBohr$ set $S$ satisfies  $(A - A) \cap S \neq \varnothing$ whenever $d^*(A) > 0$. In the language of ergodic theory, this means that every $\DiffBohr$ set is a \emph{set of recurrence}. 
We discuss this connection further in \cref{sec:relations_with_others}.  It is well-known that when $E \subseteq G$ is infinite,  $E-E$ is a set of recurrence.  The next theorem therefore shows that not every set of recurrence is $\DiffBohr$.

\begin{Maintheorem}\label{mainthm:difference_set_not_DF}
There exists an infinite set $E \subseteq \Z$ such that $E - E$ is not $\DiffBohr$.
\end{Maintheorem}
Bogolyubov's theorem (\cref{th:bogo}) can be slightly strengthened to show that if $d^*(B) > 0$ then $B - B$ is $\DiffBohr$ (for example, see \cref{cor:generalization_bogolyubov}). \cref{mainthm:difference_set_not_DF} demonstrates that we cannot drop the assumption $d^*(B) > 0$ in that statement.

\subsection{Answering questions about dense sets in Bohr compactification}

Our proof of \cref{mainthm:difference_set_not_DF} yields a stronger statement in Theorem \ref{th:ultimateCounterExample}.  
In fact, the construction in that theorem allows us to answer two questions posed by the second author in \cite{griesmer-dense-set} regarding the largeness of double sumset of the form $A + B$. Before stating the precise result, we first need some definitions.
Given a discrete abelian group $G$, the \emph{Bohr compactification of $G$} is a compact Hausdorff abelian group  $bG$ together with an injective homomorphism $\iota:G \to bG$ such that
\begin{enumerate}
	\item[(a)] $\iota(G)$ is topologically dense in $bG$;
	
	\item[(b)] for all character $\chi$ of $G$, there is a character $\tilde{\chi}$ of $bG$ such that $\tilde{\chi}\circ \iota = \chi$.
\end{enumerate}
Identifying $G$ with its image $\iota(G)$, we can consider any set $A \subseteq G$ as a subset of $bG$ and we write $\overline{A}$ for the topological closure of $A$ in $bG$. We denote by $m_{bG}$ the probability Haar measure on $bG$. For more details on Bohr compactification, see \cref{sec:bohr_compactification}.

A subset $A\subseteq G$ is called 
\begin{enumerate}
    \item[$\bullet$] \textit{syndetic} if there exists a finite set $F$ such that $A+F=G$,

    \item[$\bullet$] \emph{thick} if it intersects every syndetic set,

    \item[$\bullet$] \emph{piecewise syndetic} if it is the intersection of a thick set and a syndetic set,

    \item[$\bullet$] \emph{piecewise Bohr} if it is the intersection of a thick set and a Bohr neighborhood of some $g \in G$.
\end{enumerate}
It is obvious from the definitions that thick $\Rightarrow$ piecewise Bohr $\Rightarrow$ piecewise syndetic.
The following is {\cite[Question 5.1]{griesmer-dense-set}.
\begin{question}
\label{ques:griesmer_closure_bZ}
For $A, B \subseteq \Z$, which of the following implications hold?
\begin{enumerate}
    \item If $m_{b\Z}(\overline{A}) > 0$ and $d^*(B) > 0$, then $A + B$ is piecewise syndetic.

    \item \label{item:ques_griesmer_pwBohr} If $m_{b\Z}(\overline{A}) > 0$ and $d^*(B) > 0$, then $A + B$ is piecewise Bohr.

    \item \label{item:quest_griesmer_thick} If $\overline{A} = b\Z$ and $d^*(B) > 0$, then $A + B$ is thick.
\end{enumerate}
\end{question}
Our next theorem gives negative answers to \eqref{item:ques_griesmer_pwBohr} and \eqref{item:quest_griesmer_thick} of \cref{ques:griesmer_closure_bZ}.
\begin{Maintheorem}
\label{thm:A+B_not_pwBohr}
There exist $A, B \subseteq \Z$ with $\overline{A} = b\Z$ and $d^*(B) > 0$ such that $A + B$ is not a piecewise Bohr set (and so is not thick).
\end{Maintheorem}

\subsection{Expanding almost Bohr sets}   

F{\o}lner's theorem (\cref{thm:folner})  says that if $d^*(A) > 0$, then $A - A$ contains $B\setminus E$ where $B$ is a Bohr set and $d^*(E) = 0$.  When $S$ is $\DiffBohr$, it is natural to ask whether the Bohr structure in $A-A+S$ truly relies on $A- A$ being a difference set, or whether it depends only on the fact that $A - A$ contains most of a Bohr set. To facilitate further discussion, we first give explicit definitions. 
\begin{definition}
A set $A \subseteq G$ is an \emph{almost Bohr set} if $A = B\setminus E$ where $B$ is a Bohr set and $d^*(E) = 0$.
\end{definition}
\begin{definition}
A set $S \subseteq G$ is called \emph{$\AlmostBohrBohr$} if $A + S$ contains a Bohr set for every almost Bohr set $A$. (Here $\mathcal{A}$ stands for almost Bohr sets and $\mathcal{B}$ for Bohr sets.)
\end{definition}

F{\o}lner's theorem implies that every $\AlmostBohrBohr$
 set is also $\DiffBohr$. By \cref{cor:Specific-intro_main}, the sets $\{n^2: n \in \N\}$ and $\{p - 1: p \text{ prime}\}$ are $\DiffBohr$ in $\Z$. 
 The following characterization implies that every $\AlmostBohrBohr$ set has positive upper Banach density and as a result shows the aforementioned $\DiffBohr$ sets are not $\AlmostBohrBohr$.


\begin{Maintheorem}\label{mainthm:almost_Bohr_complete}
    For $S \subseteq G$, the following are equivalent:
    \begin{enumerate}
        \item \label{item:ABC1-intro} 
        $S$ is $\AlmostBohrBohr$.

        \item \label{item:ABC2-intro} 
        For every Bohr set $B$, $d^*(S \cap B) > 0$.

        \item \label{item:ABC3-intro} For every almost Bohr set $B$, $S \cap B \neq \varnothing$. 
    \end{enumerate} 
\end{Maintheorem}


\cref{mainthm:almost_Bohr_complete} gives us easy criteria to check for $\AlmostBohrBohr$ sets. Using it, we can recover several results about Bohr sets in sumsets. For example, in view of F{\o}lner's theorem, it is immediate to see that if $d^*(A) > 0$, then $A - A$ satisfies \eqref{item:ABC2-intro}. Therefore, we recover the qualitative statement of Bogolybov's theorem and its generalization to expressions of the form $A - A + B - B$ (see \cref{cor:generalization_bogolyubov}). \cref{mainthm:almost_Bohr_complete} also recovers our results in \cite{Griesmer-Le-Le-countable, Le-Le-compact} regarding sumsets in partition. We will discuss more about this in \cref{sec:central_intro} below.



\subsection{Central sets and partitions}
\label{sec:central_intro}
While the problem of finding Bohr sets in sumsets where the summands have positive upper Banach density has attracted much attention, the analogous question concerning partitions was little studied until recently, and the situation is less well understood. The following open question was popularized by Katznelson \cite{Katznelson_Chromatic} and Ruzsa \cite[Chapter 5]{Ruzsa-sumset_and_structure}: If $\Z = \bigcup_{i=1}^r A_i$, must $A_i - A_i$ contain a Bohr set for some $i$?
In terms of dynamical systems, this asks if every set of recurrence for minimal isometries (also known as a set of Bohr recurrence) is also a set of recurrence for minimal topological systems. See \cite{Glasscock_Koutsogiannis_Richter} for a detailed account of the history of this question and many equivalent formulations.  See \cite{griesmer-specialcases} for more equivalent formulations and resolution of some special cases.

In this setting of finding Bohr sets in partitions and in the same spirit of expanding to arbitrary abelian groups as \cref{th:main-density}, the last three authors \cite{Griesmer-Le-Le-countable} proved the following:

\begin{theorem}[{\cite[Theorem 1.4]{Griesmer-Le-Le-countable}}]
\label{th:main-partition}
Let $G$ be a discrete abelian group and let $\phi_1, \phi_2: G \rightarrow G$ be commuting homomorphisms such that $[G:\phi_j(G)]$ is finite for $j \in  \{1, 2\}$.
Then for every finite partition $G = \bigcup_{i=1}^r A_i$, there exists $i \in \{1, \ldots, r\}$ such that
\[
\phi_1(A_i) - \phi_1(A_i) + \phi_2(A_i)
\]
contains a Bohr set whose rank and radius depend only on $r$ and $[G:\phi_j(G)]$.
\end{theorem}

The proofs of Theorems \ref{th:main-density} and \ref{th:main-partition} in \cite{Griesmer-Le-Le-countable} use  commutativity of the $\phi_i$ to parametrize the equations $\phi_1(x) + \phi_2(y) + \phi_3(z) = 0$ and $\phi_1(x) - \phi_1(y) + \phi_2(z) = 0$, respectively. 
The commutativity assumption is essential for those proofs, but it does not seem necessary for the results to hold. This observation motivated us to ask the following question in that paper.


\begin{question}[see {\cite[Question 11.1]{Griesmer-Le-Le-countable}}]
\label[question]{ques:drop_commuting}
Can the assumption that the $\phi_i$ commute in \cref{th:main-density} and \cref{th:main-partition} be dropped?
\end{question}

Using the tools developed in our study of $\AlmostBohrBohr$ sets, we offer a partial answer to \cref{ques:drop_commuting}, presented in the form of \cref{thm:discrete_group_central_sets} below. \cref{thm:discrete_group_central_sets} also answers another question, which we now introduce.
In \cref{th:main-partition}, the selection of the set $A_i$ may depend on the homomorphisms $\phi_1, \phi_2$. In the setting when the ambient group being $\Z$, the last author and Wathodkar \cite{Le-Wathodkar} proved a rather surprising result that there is a single set $A_i$ that works for all homomorphisms.
\begin{theorem}[\cite{Le-Wathodkar}]
\label{thm:le-wathodkar}
    For any partition $\Z = \bigcup_{i=1}^r A_i$, there is an $i$ such that $s_1 A_i - s_1 A_i + s_2A_i$ contains a Bohr set for any $s_1, s_2 \in \Z \setminus \{0\}$.
\end{theorem}
\cref{thm:le-wathodkar} was first proved by the third and fourth authors \cite{Le-Le-compact} in the case $s_1 = 1$.
In the same vein of extending $\Z$ to arbitrary discrete abelian group $G$, \cref{thm:le-wathodkar} gives rise to the following question:
\begin{question}\label[question]{ques:one_set_works}
Let $G = \bigcup_{i=1}^r A_i$ be a partition of a discrete abelian group. Does there exist an $i$ such that $\phi_1(A_i) - \phi_1(A_i) + \phi_2(A_i)$ contains a Bohr set for all homomorphisms $\phi_1, \phi_2: G \to G$ satisfying $[G: \phi_j(G)]$ is finite for $j \in \{1, 2\}$?
\end{question} 

As a corollary of \cref{mainthm:almost_Bohr_complete}, \cref{thm:discrete_group_central_sets} below gives a positive answer to 
\cref{ques:one_set_works}. As previously mentioned, it also gives a partial answer to  \cref{ques:drop_commuting} by showing that the commuting assumption in \cref{th:main-partition} can be dropped.

Theorem \ref{thm:discrete_group_central_sets} concerns central sets, which we will define in Section \ref{sec:def_notation}. Central sets were first introduced in $\Z$ by Furstenberg \cite{Furstenberg81} and later generalized to arbitrary semigroups by Glasner \cite{glasner_1980}  and Bergelson, Hindman, Weiss \cite{Bergelson_Hindman90}.
Many results in additive combinatorics and Ramsey theory rely on the fact that, whenever a group is partitioned into finitely many subsets, at least one set is central (for example, see \cite{
Bergelson_Hindman90, Bergelson_Hindman94, Bergelson_Hindman01, Hindman_Strauss98}). We will prove that every central set is $\AlmostBohrBohr$ and so obtain the following:


\begin{Maintheorem}
\label{thm:discrete_group_central_sets}
    If $A \subseteq G$ is a central set, then for all (not necessarily commuting) homomorphisms $\phi_1, \phi_2: G \to G$ satisfying $[G:\phi_i(G)]$ is finite, the sumset
    \[
        \phi_1(A) - \phi_1(A) + \phi_2(A)
    \]
    contains a Bohr set.
\end{Maintheorem}

\subsection{Relationships with other dynamical families and applications to pointwise recurrence}
\label{sec:relations_with_others}

Furstenberg's seminal work on single and multiple recurrence \cite{Furstenberg77, Furstenberg81} lead to analysis of various recurrence properties of subsets of $\mathbb Z$ and other groups.   These properties include recurrence, strong recurrence, nice recurrence, and the van der Corput property; see \cref{sec:def_notation} for definitions.
By definition, nice recurrence implies strong recurrence, but the converse is false \cite{Forrest90, McCutcheon93}. The notions of van der Corput sets, strong recurrence, and nice recurrence are all strictly stronger than recurrence. Our next theorem shows that $\AlmostBohrBohr$ is stronger than all of them.

\begin{Maintheorem}\label{thm:ABB_vanderCorput_optimal_recurrence}
If $A \subseteq G$ is $\AlmostBohrBohr$, then $A$ is both a van der Corput set and a set of nice recurrence.
\end{Maintheorem}

None of the converse directions in \cref{thm:ABB_vanderCorput_optimal_recurrence} are true: the set $\{n^2: n \in \N\}$ is van der Corput and is set of nice recurrence in $\Z$, but not $\AlmostBohrBohr$. \cref{thm:ABB_vanderCorput_optimal_recurrence} should be compared with \cref{prop:DB-not-strong-recurrence} below which says that there exists a $\DiffBohr$ set which is not a set of strong recurrence. This difference again demonstrates that the notion of $\AlmostBohrBohr$ sets is significantly stronger than $\DiffBohr$.

Another notion of recurrence is pointwise recurrence. A set $A \subseteq G$ is called a \emph{set of pointwise recurrence} if for all minimal $G$-systems $(X,(T_g)_{g \in G})$, all points $x \in X$, and all neighborhood $U$ of $x$, there exists $a \in A$ such that $T_g x \in U$. In the context $G = \Z$, this notion of recurrence was first formally defined by Host, Kra, and Maass \cite{HostKraMaass2016} and was explored further by Glasscock and the third author \cite{Glasscock-Le-syndetic, glasscock_le_pointwise}. 
It is shown in \cite{Glasscock-Le-syndetic} that pointwise recurrence is strictly stronger than recurrence. Our next result shows that pointwise recurrence is even stronger than $\AlmostBohrBohr$.

\begin{Maintheorem}
\label{prop:pointwise_recurrence_is_ABB}
Every set of pointwise recurrence in $\Z$ is an $\AlmostBohrBohr$ set.
\end{Maintheorem}
The converse of \cref{prop:pointwise_recurrence_is_ABB} is false: while every set of pointwise recurrence in $\Z$ is piecewise syndetic \cite{Glasner_Weiss89}, we will show in \cref{mainthm:counter_for_almost_Bohr_0}
below that there exists an $\AlmostBohrBohr$ set in $\Z$ that is not piecewise syndetic. 

As an immediate corollary of Theorems \ref{thm:ABB_vanderCorput_optimal_recurrence} and \ref{prop:pointwise_recurrence_is_ABB}, we obtain additional, previously unknown properties of sets of pointwise recurrence in $\Z$.
\begin{Maintheorem}
Every set of pointwise recurrence in $\Z$ is a van der Corput set and a set of nice recurrence.
\end{Maintheorem}

\subsection{Other properties of \texorpdfstring{$\AlmostBohrBohr$}{(AlmostBohr, Bohr) expanding} sets}

The examples of $\AlmostBohrBohr$ sets discussed so far include central sets,  difference sets of the form $A - A$ where $d^*(A) > 0$, and sets of pointwise recurrence in $\Z$. A common feature among these sets is that they are piecewise syndetic. This appears to be somewhat necessary as \cref{mainthm:almost_Bohr_complete} implies that an $\AlmostBohrBohr$ set must have positive upper Banach density. That being said, the gap between positive upper Banach density and piecewise syndeticity raises a natural question: is every $\AlmostBohrBohr$ set piecewise syndetic? Furthermore, every example above is a set of multiple recurrence, and except for the sets of pointwise recurrence they are all $IP$ sets. See \cref{sec:def_rec} for the definition of ``set of multiple recurrence,'' ``$IP$ set,'' etc. Theroem \ref{mainthm:counter_for_almost_Bohr_0} provides an example of a $\AlmostBohrBohr$ set having none of these properties.

\begin{Maintheorem}\label{mainthm:counter_for_almost_Bohr_0}
There exists $S \subseteq \Z$ such that 
\begin{enumerate}
    \item $A + S = \Z$ for every almost Bohr set $A$, 
    \item $S$ is not piecewise syndetic,
    \item $S$ is not a set of $2$-recurrence,
    \item and $S$ is not an $IP_0$ set.
\end{enumerate}
\end{Maintheorem}

We also prove the following:
\begin{enumerate}
    \item As a family of subsets of $G$, the collection of $\AlmostBohrBohr$ sets has the Ramsey property (partition regularity) (see \cref{prop:partition_regular_almostBohr}).

    \item $\AlmostBohrBohr$ sets do not necessarily contain a difference set (see \cref{prop:dcps_not_difference}).
\end{enumerate}

\subsection{Diagram of implications}
In Figure \ref{figure:relations_intersective_prime} below, we summarize the relationships among the classes of sets studied in our paper. An arrow from (P) to (Q) indicates that every set having property (P) has property (Q), crossed out arrows indicate that at least one set with property (P) does not have property (Q).  Labels indicate proofs given in the present article.

\begin{figure}[ht]\label{fig:PrimeRelations}
    \centering
\adjustbox{scale=0.85,center}{
    \begin{tikzcd}[row sep= large]
         & \fbox{\begin{tabular}{c}
             Pointwise recurrence
        \end{tabular}} 
        \ar[d, rightarrow, "Thm. \ref{prop:pointwise_recurrence_is_ABB}"] 
        & & 
        \\
        \fbox{\begin{tabular}{c}
             Piecewise syndetic
        \end{tabular}}
        & 
        \fbox{\begin{tabular}{c}
             $\AlmostBohrBohr$
        \end{tabular}}
        \ar[l, rightarrow, "//" anchor=center, "Thm. \ref{mainthm:counter_for_almost_Bohr_0}"]
        \ar[dl, rightarrow, "//" anchor=center, "Thm. \ref{mainthm:counter_for_almost_Bohr_0}"]
        \ar[d, rightarrow]
        \ar[dr, rightarrow,bend right = 15, looseness=1.4, "Thm. \ref{thm:ABB_vanderCorput_optimal_recurrence}"]
        \ar[drr, rightarrow, bend left = 50, looseness=1, "Thm. \ref{thm:ABB_vanderCorput_optimal_recurrence}"]
        & 
        \fbox{\begin{tabular}{c}
             Support $\KMF$ mean
        \end{tabular}} 
        \ar[dl, rightarrow, bend right = 15, looseness=1.4, "Thm. \ref{th:SpectralSumDifference_intro}"]
        \ar[d, rightarrow, "Prop. \ref{prop:kmf_imply_vanderCorput}"]
        \ar[dr, rightarrow, "Prop. \ref{prop:kmf_imply_vanderCorput}"]
        & 
        \\
        \fbox{\begin{tabular}{c}
             Multiple recurrence
        \end{tabular}} 
        \ar[dr, rightarrow]
        & 
        \fbox{\begin{tabular}{c}
             $\DiffBohr$
        \end{tabular}}
        \ar[d, rightarrow]
        &
        \fbox{\begin{tabular}{c}
             van der Corput
        \end{tabular}}
        \ar[l, rightarrow, "//" anchor=center, "Thm. \ref{mainthm:difference_set_not_DF}"]
        \ar[dl, rightarrow]
        &
        \fbox{\begin{tabular}{c}
        Set of nice recurrence
        \end{tabular}}
        \ar[dll, rightarrow]
        \\
        & \fbox{\begin{tabular}{c}
             Set of recurrence
        \end{tabular}}
         & &
    \end{tikzcd}
    }

    \caption{Implications between classes of sets studied in our paper.}
    \label{figure:relations_intersective_prime}
\end{figure}
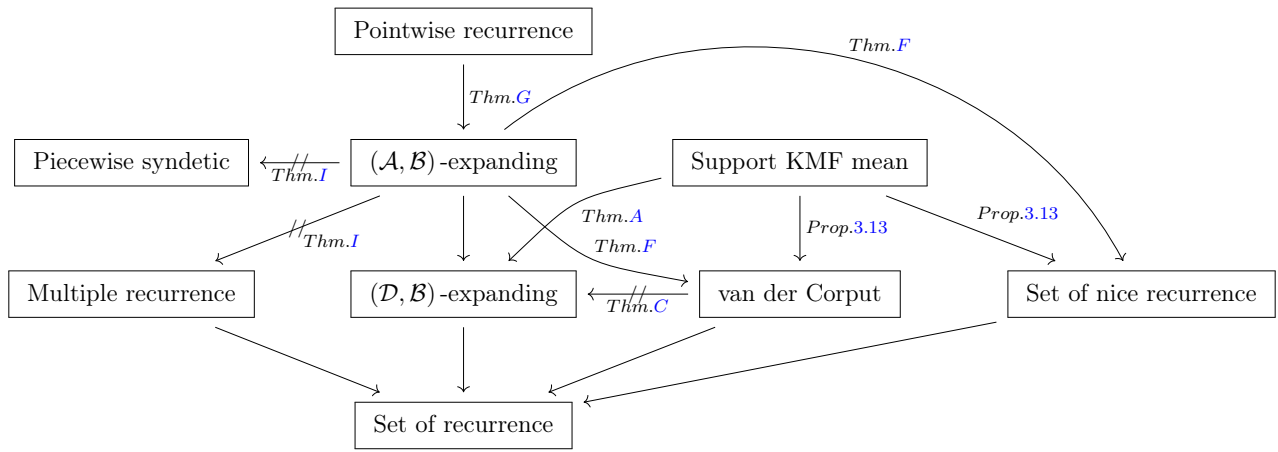

\subsection{Acknowledgments}
The fourth author is supported by NSF Grant DMS-2246921 and a Travel Support for Mathematicians gift from the Simons Foundation.

\section{Definitions and notations}
\label{sec:def_notation}

\subsection{Dynamical systems and notions of recurrence}
\label{sec:def_rec}

 Let $G$ be a discrete abelian group. A \emph{topological $G$-system} is a pair $(X, (T_g)_{g \in G})$ where $X$ is a compact metric space and $(T_g)_{g\in G}$ is an action of $G$ on $X$ by homeomorphisms. A topological $G$-system $(X, (T_g)_{g \in G})$ is \emph{minimal} if for every $x \in X$, $\{T_g x: g \in G\}$ is dense in $X$.
A \emph{measure preserving $G$-system} is a tuple $(X, \mu, (T_g)_{g \in G})$  where $(X,\mu)$ is a probability measure space and $(T_g)_{g\in G}$ is an action of $G$ on $X$ satisfying $\mu(E) = \mu(T_g E)$ for all $\mu$-measurable sets $E \subseteq X$.

The following notions of recurrence are well-studied in topological dynamics and ergodic theory. For $A \subseteq G$, we say that
\begin{enumerate}
    \item $A$ is a \emph{set of recurrence}  if for every measure preserving $G$-system $(X, \mu, (T_g)_{g \in G})$ and every $E \subseteq X$ such that $\mu(E) > 0$, there exists $a \in A$ such that $\mu(E \cap T_a E) > 0$; 

    \item $A$ is a set of \emph{strong recurrence} if for every measure preserving $G$-system $(X, \mu, (T_g)_{g \in G})$ and every $E \subseteq X$ such that $\mu(E) > 0$, there exists $c>0$ and infinitely many $a\in  A$ such that $\mu(E \cap T_{a} E) > c$;

    \item $A$ is a set of \emph{nice recurrence} if for every measure preserving $G$-system $(X, \mu, (T_g)_{g \in G})$, every $E \subseteq X$ such that $\mu(E) > 0$, and every $\epsilon > 0$, there exists $a \in A$ such that $\mu(E \cap T_a E) > \mu(E)^2 - \epsilon$;

    \item $A$ is a \emph{van der Corput set} if for every positive finite Radon measure $\sigma$ on $\widehat{G}$, $\hat{\sigma}(a)=0$ for all $a\in A$ implies $\sigma(\{\chi_0\})=0$, where $\chi_0\in \widehat{G}$ is the trivial character: $\chi_0(g) = 1$ for all $g \in G$.  
    
    Equivalently, $A$ is van der Corput if and only if it is a \textit{set of operatorial recurrence}, i.e., if $\mathcal{H}$ is any Hilbert space, $U$ is any unitary representation of $G$ on $\mathcal{H}$, if $x \in \mathcal{H}$ satisfies 
    $\langle x, U_a x\rangle = 0$ for every $a \in A$, then $x$ is orthogonal to the subspace of $\mathcal{H}$ of $U$-invariant vectors. (For the proof of this equivalent characterization in case $G$ is countable, see for example \cite[Theorem 4.3]{Farhangi_Tucker-Drob-asymptotic}. The proof of uncountable groups is the same.) 
    

    \item $A$ is a \emph{set of $k$-recurrence} if for every measure preserving $G$-system $(X, \mu, (T_g)_{g \in G})$, every $E \subseteq X$ such that $\mu(E) > 0$, there exists $a \in A$ such that
    \[
        \mu(E \cap T_a E \cap \cdots \cap T_{ka} E) > 0.
    \]

    \item $A$ is a \emph{set of multiple recurrence} if it is a set of $k$-recurrence for all $k \in \N$.

     \item $A$ is a \emph{set of pointwise recurrence} if for every minimal topological $G$-system $(X,(T_g)_{g \in G})$, every $x \in X$, and neighborhood $U$ of $x$, there exists $a \in A$ such that $T_a x \in U$.
\end{enumerate} 

Here are the previous known relations between these forms of recurrence: nice recurrence $\Rightarrow$ strong recurrence $\Rightarrow$ recurrence (by definitions); van der Corput $\Rightarrow$ recurrence \cite{Kamae_France78}; and in case $G = \Z$, pointwise recurrence $\Rightarrow$ multiple recurrence $\Rightarrow$ recurrence \cite{Glasscock-Le-syndetic}. The converses of these implications are known to be false in case $G = \Z$ \cite{Bourgain87, Forrest90, Glasscock-Le-syndetic}. 

\subsection{Means and density}
\label{sec:prelim_density_syndetic}




Given a function $f: G \to \C$ and $g\in G$, we define the translate $\tau_g f$ to be the function given by $(\tau_gf)(x):=f(x-g)$.  
We write $\ell^\infty(G)$ for the Banach space of bounded complex-valued functions on $G$ equipped with the supremum norm: $\|f\|_\infty := \sup\{|f(g)|:g\in G\}$.  

A \emph{mean} on $\ell^{\infty}(G)$ is a linear functional $m:\ell^{\infty}(G)\to \mathbb C$ satisfying $m(1_G)=1$ and $m(f)\geq 0$ whenever $f$ is real-valued and $\inf f\geq 0$.

An \emph{invariant mean} on $\ell^{\infty}(G)$ is a mean $m$ satisfying $m(\tau_g f)=m(f)$ for all $g\in G$.   
The set of invariant means on $G$ will be denoted $\mathcal M_\tau(G)$. In our setting ($G$ is a discrete abelian group), $\mathcal M_\tau(G)$ is always nonempty (for example, see Proposition 0.15 of \cite[page 14]{Paterson88}).
The \emph{upper Banach density} of a set $A\subseteq G$ is $d_G^*(A):=\sup\{m(1_A): m\in \mathcal M_\tau(G)\}$. If the ambient group $G$ is clear from context, we simply write $d^*(A)$. Every piecewise syndetic set has positive upper Banach density every thick set has upper Banach density $1$; see Lemmas 12.2 and 12.4 of \cite{Griesmer_DiscreteSumsets}.

\subsection{Central sets and \texorpdfstring{$IP$}{IP} sets}
\label{sec:central_ip_sets}

Given a $G$-system $(X, (T_g)_{g \in G})$ with the metric $d$ on $X$, two points $x, y \in X$ are \emph{proximal} if for every $\epsilon > 0$, $\{g \in G: d(T_g x, T_g y) < \epsilon\}$ is thick. Roughly speaking, two points are proximal if there are infinitely many times their orbits get arbitrarily close together. A point $y \in X$ is \emph{uniformly recurrent} if for every neighborhood $U$ of $y$, the set $\{g \in G: T_g y \in U\}$ is syndetic. Equivalently, the orbit closure $\overline{\{T_g y: g \in G\}}$ forms a minimal subsystem of $(X, T)$; see \cite[Theorems 1.15 and 1.17]{Furstenberg81} for a proof.

Central sets in $\N$ were first introduced by Furstenberg \cite[Definition 8.2]{Furstenberg81} which can be generalized to arbitrary semigroup:
A set $C \subseteq G$ is \emph{central} if there exists a $G$-system $(X, (T_g)_{g \in G})$, two proximal points $x, y \in X$ in which $y$ is uniformly recurrent, a neighborhood $U$ of $y$ such that
\[
    C \supseteq \{g \in G: T_g x \in U\}.
\]
It was shown by Glasner \cite{glasner_1980}  and Bergelson, Hindman, Weiss \cite{Bergelson_Hindman90} that central sets can be defined in terms of the the topological semigroup $\beta G$, the Stone-\v{C}ech compactification of $G$, viewed as the space of ultrafilters on $G$. A set $C \subseteq G$ is called a \emph{central set} if it is a member of a minimal idempotent in $\beta G$. 

Central sets are piecewise syndetic, and centrality is a partition regular property (if $B \cup C$ is central, then $B$ or $C$ must be  central). For an overview of central sets, see \cite{Hindman_Strauss98}.

For a set $A \subseteq G$,
\begin{enumerate}
    \item $A$ is an \emph{$IP_0$ set} if for every $k \in \N$, there exists $g_1, \ldots, g_k \in G \setminus \{0\}$ such that $\sum_{n \in F} g_n \in A$ for any set $F \subseteq \{1, \ldots, k\}$;

    \item $A$ is an \emph{$IP$ set} if there exists a sequence $(g_n)_{n \in \N} \subseteq G \setminus \{0\}$ such that $\sum_{n \in F} g_n \in A$ for all finite sets $F \subseteq \N$.
\end{enumerate}
Evidently, every $IP$ set is an $IP_0$ set. We can also define $IP$ sets in terms of $\beta G$: a set is $IP$ if it belongs to an idempotent in $\beta G$. As a result, every central set is an $IP$ set. The dual notion of $IP$ sets is $IP^*$ sets: a set is $IP^*$ if it intersects every $IP$ set in $G$, or equivalently, it belongs to every idempotent in $\beta G$. It is shown by Furstenberg and Katznelson \cite{Furstenberg_Katznelson85} that every $IP_0$ set in $\Z$ is a set of multiple recurrence.

\subsection{Bohr sets and Bohr compactification} 

\label{sec:bohr_compactification}

The group of characters of $G$ is called the dual of $G$ and is  denoted by $\widehat{G}$. Recall the definition of Bohr sets in \cref{sec:history}: 
A \emph{Bohr set of rank $k$ and radius $\eta$} is a set of the form  
\begin{equation*} 
    \{ g \in G : | \chi(g)-1 | < \eta \textup{ for all } \chi \in \Lambda\}
\end{equation*} 
where $\Lambda \subseteq \widehat{G}$ of size $k$.
For $g \in G$ and a Bohr set $B$, the set $g + B$ is called \emph{Bohr neighborhood of $g$}. We have the following well-known fact whose short proof is included for completeness. 
\begin{lemma}\label[lemma]{lem:bohr_syndetic_set}
   Every Bohr set is syndetic, and so has positive upper Banach density.
\end{lemma}
\begin{proof}
Fix $d\in \mathbb N$, a set of characters $\Lambda = \{\chi_1,\dots,\chi_d\}$ of $G$, and $\eta>0$.  We will show that $G$ is covered by finitely many translates of $B:=\Bohr(\Lambda;\eta)$. 

Define $\chi: G \to (S^1)^d$ by $\chi(g) = (\chi_1(g), \ldots, \chi_d(g))$. Let $U = \{z\in S^1 : |z-1| < \eta\}$. The closure $\overline{\chi(G)}$ of the image of $\chi$ is a compact subset of $(S^1)^d$, while the cartesian power $U^d$ is open, so $\chi(G)$ can be covered by finitely many translates of $U^d$ by elements of $\chi(G)$. This implies that $G$ is covered by finitely many sets of the form $\chi^{-1}(\chi(g) U^d)$, and each such set is a translate of $B$.
\end{proof}

 A \emph{trigonometric polynomial} $p:G\to \mathbb C$ is a function of the form $p=\sum_{j=1}^d c_j\chi_j$, where $d\in \mathbb N$, $c_1,\dots, c_d\in \mathbb C$, and $\chi_j\in \widehat{G}$. A function $f: G \to \C$ is called a \emph{uniformly almost periodic} function if it is a uniform limit of a sequence of trigonometric polynomials.

The \emph{Bohr compactification of $G$} is a compact Hausdorff abelian group  $bG$ together with a one-to-one homomorphism $\iota:G \to bG$ such that
\begin{enumerate}
	\item[(i)] $\iota(G)$ is topologically dense in $bG$;
	
	\item[(ii)] for all $\chi\in \widehat{G}$, there is a character $\tilde{\chi}\in \widehat{bG}$ such that $\tilde{\chi}\circ \iota = \chi$.
\end{enumerate}
We will identify $G$ with its image $\iota(G)$, and consider $G$ as a topologically dense subgroup of $bG$, so that each $\chi\in \widehat{G}$ is the restriction of a continuous $\tilde{\chi}\in \widehat{bG}$: $\chi = \tilde{\chi}|_G$.  

The group $bG$ may be constructed as the Pontryagin dual of $\widehat{G}_d$, where $\widehat{G}_d$ is the group $\widehat{G}$ with the discrete topology instead of the usual topology.  In this construction, the embedding map $\iota$ is given by $\iota(g):=e_g$, where $e_g(\chi)$ is the evaluation map $e_g(\chi) =\chi(g)$. We write $\mu_{bG}$ for Haar probability measure on $bG$. See Section 1.8 of \cite{Rudin62} or Section 4.7 of \cite{Folland_Abstract} for more on $bG$.

The following is the special case of \cite[Theorem 4.79]{Folland_Abstract} for discrete abelian groups.

\begin{proposition}\label[proposition]{th:BohrCompactificationAndAP} 
	If $f$ is a bounded  function on $G$, the following are equivalent: 
	\begin{enumerate}
        \item $f$ is uniformly almost periodic.	
        
		\item $f$ is the restriction to $G$ of a continuous function on $bG$.
	\end{enumerate}
\end{proposition}

Likewise, we have the following characterization of Bohr sets. In the next proposition and the remainder of the paper, $\Re z$ denotes the real part of $z$.
\begin{proposition}\label[proposition]{lem:BohrFourierEquivalent}
Let $S\subseteq G$. The following are equivalent:
\begin{enumerate}
    \item\label{item:Bohr0} $S$ contains a Bohr set.

    \item\label{item:Bohr0UAP} There is a uniformly almost periodic function $\phi: G\to \mathbb C$ such that $\Re \phi(0)>0$ and $\{g \in G: \Re \phi(g) >0 \}\subseteq S$. 

    \item\label{item:PositiveUAP} There is a uniformly almost periodic function $\phi:G\to [0,1]$ supported on $S$ with $\phi(0)>0$.
\end{enumerate}
\end{proposition}

\begin{proof}
\noindent (\ref{item:Bohr0})$\implies$(\ref{item:Bohr0UAP}). Assume $S$ contains the Bohr set $B:=\Bohr(\chi_1,\dots,\chi_d;\eta)$ and let $\phi=\eta^2-\sum_{j=1}^d |\chi_j-1|^2$.  Then $\phi$ is a uniformly almost periodic function, and $\phi(g)>0$ implies $g\in B$.  Since $B\subseteq S$, this means $\{g \in G: \Re \phi(g) >0\} \subseteq S$.

\noindent (\ref{item:Bohr0UAP})$\implies$(\ref{item:Bohr0}). Let $\phi$ be uniformly  almost periodic and satisfy $\phi(0)>0$.  Let $p=\sum_{j=1}^{d} c_j\chi_j$ be a trigonometric polynomial with $\|\phi-p\|_\infty < \frac{1}{2}|\phi(0)|$.  Let $M=1+\max_j|c_j|$, and let $\eta=\frac{1}{1+2dM}|\phi(0)|$.  Let $B= \Bohr(\chi_1,\dots,\chi_d;\eta)$.  We will prove that $B\subseteq \{g \in G: \Re \phi(g) >0\}$.  To prove this, let $g\in B$. Then 
\[
    |p(g)-p(0)|=\left|\sum_{j=1}^d c_j\chi_j(g)-c_j\chi_j(0)\right|\leq \sum_{j=1}^d |c_j||\chi_j(g)-1| \leq dM\eta<\frac{1}{2}|\phi(0)|.
\]  
For such $g$, we have 
\[
    |\phi(g)-\phi(0)|\leq |\phi(g)-p(g)|+|p(g)-p(0)|<|\phi(0)|.
\]
Thus $|\phi(g)-\phi(0)|<|\phi(0)|$.  Since $\phi(0)>0$, this implies $\Re \phi(g)>0$.  

See \cite[Lemma 12.6]{Griesmer_DiscreteSumsets} for a proof of (\ref{item:Bohr0UAP})$\iff$(\ref{item:PositiveUAP}), which is similar to the above proof.
\end{proof}

\section{Bohr structure in \texorpdfstring{$A-A+S$}{A-A+S}}
\label{sec:proof_of_TheoremA}

The goal of this section is to prove \cref{th:SpectralSumDifference_intro}. Before going to the proofs, we first need some preliminary definitions and lemmas.

\subsection{Positive definite functions and Fourier transforms of measures}

\label{sec:positive_definite}

Let $G$ be a discrete abelian group.  A function $\phi:G\to \mathbb C$ is \emph{positive definite} 
if for every function $c:G\to \mathbb C$ with finite support, 
\[
    \sum_{g,g'\in G} c(g)\overline{c(g')} \phi(g-g')\geq 0.
\]
The most relevant examples are the \emph{correlation sequences} given by an invariant mean $m\in \mathcal M_{\tau}(G)$ and a function $f\in \ell^\infty(G)$:
\[
    \phi_f(g):= m(f \cdot \tau_g \overline{f}).
\]
Another class of examples is given by a measure-preserving system $(X, \mu, (T_g)_{g \in G})$ and $A \subseteq X$:
\[
    \phi_A(g):= \mu(A \cap T_g A).
\]
Another class of examples comes from linear isometries of (pseudo-)metrics induced by a semidefinite Hermitian form positive Hermitian form $\langle \cdot,\cdot \rangle $ on a vector space $V$. Given such a form and an action $U$ of $G$ on $V$ preserving $\langle \cdot ,\cdot \rangle$, we define $\phi_w:G\to \mathbb C$  by $\phi_w(g):=\langle w, U_g w\rangle$.  To see that $\phi_w$ is positive definite, let $c:G\to \mathbb C$ have finite support. Then
\begin{align*}
    \sum_{g,g'\in G} c(g)\overline{c(g')}\phi_w(g-g')&=\sum_{g,g'\in G} c(g)\overline{c(g')}\langle w, U_{g-g'} w\rangle\\
    &=\sum_{g,g'\in G} c(g)\overline{c(g')}\langle U_{-g} w, U_{-g'} w\rangle \\
    &=  \Big\langle \sum_{g\in G} c(g)U_{g}^{-1}w, \sum_{g'\in G} c(g')U_{g'}^{-1}w\Big\rangle\\
    &\geq 0.
\end{align*}

\begin{lemma}\label[lemma]{lem:HilbertianInequality}
Let $\langle\cdot, \cdot \rangle$ be a positive semidefinite Hermitian form on a vector space $V$ with associated seminorm $\|\cdot \|$.  Let $U$ be an action of $G$ on $V$ preserving $\langle \cdot, \cdot \rangle$.  Assume $v \in V$ is $U$-invariant and $\|v\|=1$.  Let $w\in V$ and $c= \langle w,v\rangle$.  Then for every $m\in \mathcal M_\tau(G)$, $m(\phi_{w})\geq |c|^2$.
\end{lemma}

\begin{proof}
Let $z= w-cv$.  Then
\begin{align*}
    \phi_{z}&=\langle w-cv,U_g(w-cv)\rangle \\
    &= \langle w,U_gw\rangle - c\langle v,U_gw\rangle - \overline{c}\langle w, U_g v\rangle + |c|^2 \langle v,U_g v\rangle\\
    &= \langle w,U_gw\rangle -c\langle v,w\rangle - \overline{c\langle v,w\rangle} + |c|^2 \langle v, v\rangle && \text{(since $\langle v,U_gw\rangle = \langle U_g^{-1} v,w\rangle= \langle v,w\rangle$)}\\
    &= \phi_{w}-c\bar{c}-\bar{c}c +|c|^2\\
    &= \phi_{w} -|c|^2.
\end{align*}
So $\phi_{z} = \phi_{w}-|c|^2$.  Since $\phi_{z}$ is positive definite, we have $m(\phi_{z})\geq 0$, which implies $m(\phi_{w})\geq |c|^2$.

\end{proof}

The following result is often called the Bochner-Herglotz theorem; it connects a positive definite functions on $G$ with Fourier transforms of  measures on $\widehat{G}$. 
\begin{theorem}[see {\cite{Rudin62}} or {\cite[p.103]{Folland_Abstract}}]
\label{th:B-H}
  Let $G$ be a discrete abelian group and $\phi$ a positive definite function on $G$.  Then there is a positive finite
Radon measure $\sigma$ on $\widehat{G}$ such that $\phi = \hat{\sigma}$, that is for all $g \in G$,
\[
    \phi(g) = \int_{\widehat{G}} \overline{\chi}(g) \ d \sigma(\chi).
\]
\end{theorem}

The next lemma is well-known.

\begin{lemma}[see {\cite{EberleinFS, Griesmer_DiscreteSumsets}}]
\label[lemma]{lem:meanOfSigmaHat}
    Let $G$ be a discrete abelian group, $m\in \mathcal M_\tau(G)$, and $\sigma$ a positive finite Radon measure on $\widehat{G}$. Then 
    \begin{enumerate}
\item\label{item:Eb1}$m(\hat{\sigma})=\sigma(\{\chi_0\})$, where $\chi_0\in \widehat{G}$ is the trivial character;  
        \item\label{item:Eb2} $m(|\hat{\sigma}|^2)= \sum_{\chi\in \widehat{G}} |\sigma(\{\chi\})|^2$.
    \end{enumerate}
\end{lemma}
Part (\ref{item:Eb2}) is \cite[Theorem 1]{EberleinFS}.
Part (\ref{item:Eb1}) is proved in 
the course of the proof of the unnumbered lemma in \cite[page 311]{EberleinFS}. It can also be deduced by combining Theorem 3.5 and Lemma 9.5 in \cite{Griesmer_DiscreteSumsets}.






\subsection{\texorpdfstring{$\KMF$}{KMF} means} 

\label{sec:kmf_mean_def}

Let $m$ be a (not necessarily invariant) mean on $\ell^{\infty}(G)$ and $S\subseteq G$. We say $m$ is \emph{supported}  $S$ if $m(1_S) = 1$, where $1_S$ is the indicator function of $S$. 
The \emph{spectrum} of $m$ is 
\[
    \spec(m):=\{\chi\in\widehat{G}: m(\overline{\chi})\neq 0\}.
\]
For $f\in \ell^{\infty}(G)$, define the function $m*f:G\to \mathbb C$ by 
\[
    (m*f)(g):=m(\tau_{g}f ).
\]

\begin{lemma}\label[lemma]{lem:ConvolutionSupport}
    If $m$ is supported on $S$, $f$ is a real-valued function on $G$ and $C := \{g \in G: f(g) > 0\}$, then $\{g \in G: m * f (g) > 0\} \subseteq S - C$.
\end{lemma}
\begin{proof}
If $h \in G$ satisfies $m*f(h)>0$, then $m(\tau_{h}f) > 0$ and so 
\[
    S\cap \{g: \tau_{h}f(g) > 0\} = S \cap \{g: f(g - h) > 0\} \neq \varnothing.
\]
This implies $f(s - h)>0$ for some $s\in S$. It follows that  $s -  h \in C$, and therefore $h \in S - C$.
\end{proof}

\begin{definition}\label{def:Annihilates}
  Let $m$ be a mean on $\ell^\infty(G)$.  We say that
  \begin{enumerate}

\item  \label{item:anihilate_continuous_measures} $m$ \emph{annihilates continuous measures on $\widehat{G}$} if $m(|\hat{\sigma}|^2)=0$ for every continuous positive Radon measure $\sigma$ on $\widehat{G}$;

\item \label{item:massively_accumulate} $m$ \emph{massively accumulates at $0$ in $bG$} if $m(1_B)>0$ for every Bohr set $B \subseteq G$;

\item $m$ is a \emph{$\KMF$ mean} if $m$ satisfies both \eqref{item:anihilate_continuous_measures} and \eqref{item:massively_accumulate}.
  \end{enumerate}
\end{definition}

\begin{remark}
Every invariant mean is a $\KMF$ mean. Indeed, let $m$ be an invariant mean and $B$ be a Bohr set. By \cref{lem:bohr_syndetic_set}, $B$ is syndetic and so there are $g_1, \ldots, g_k \in G$ such that $G = \bigcup_{i=1}^k (B + g_i)$. We have $m(1_G) = 1$ and so $m(1_{B + g_i}) \geq 1/k$ for some $i$. Since $m$ is invariant, it follows that $m(1_B) \geq 1/k > 0$. Thus $m$ massively accumulates at $0$ in $b G$. The fact that $m$ annihilates continuous measures follows from Lemma \ref{lem:meanOfSigmaHat} (ii).
\end{remark}

The next lemma says that we can replace $m(|\hat{\sigma}|^2) = 0$ with $m(\hat{\sigma}) = 0$ in Definition \ref{def:Annihilates}.
\begin{lemma}\label[lemma]{lem:AnnihilateEquiv}
A mean $m$ on $\ell^{\infty}(G)$ annihilates continuous measures if and only if $m(\hat{\sigma})=0$ for every continuous positive Radon measure $\sigma$ on $\widehat{G}$.
\end{lemma}

\begin{proof}

The ``only if'' implication follows from Cauchy-Schwarz: 
\[
    |m(\hat{\sigma})|\leq m(|\hat{\sigma}|)=m(|\hat{\sigma}1_G|)\leq m(|\hat{\sigma}|^2)^{1/2}m(1_G)^{1/2}.
\]
Thus $m(|\hat{\sigma}|^2)=0$ implies $m(\hat{\sigma})=0$.

To prove the ``if'' implication, assume $m(\hat{\sigma}) = 0$ for every continuous positive Radon measure on $\widehat{G}$. Let $\sigma$ be a such a measure. Our goal is to show $m(|\hat{\sigma}|^2) = 0$. Define the measure $\sigma * \sigma$ on $\widehat{G}$ by 
\[
    \int_{\widehat{G}} f \ d(\sigma*\sigma) = \int_{\widehat{G}^2} f(\chi \overline{\psi}) \ d \sigma(\chi) d \sigma(\psi) \text{ for } f \in C(\widehat{G}). 
\]
Continuity of $\sigma*\sigma$ follows from continuity of $\sigma$ and  Theorem 19.16 of \cite[p.~271]{HewittRoss_Abstract}. 
    
    By our assumption, $m(\widehat{\sigma * \sigma}) = 0$.
For $g \in G$, let $e_g: \widehat{G} \to \C$ denote the evaluation function, $e_g(\chi) := \chi(g)$. We have
\begin{multline*}
    \widehat{\sigma * \sigma}(g) = \int_{\widehat{G}} \overline{\chi}(g) \ d (\sigma * \sigma)(\chi) = \int_{\widehat{G}} e_{-g}(\chi) \ d (\sigma * \sigma)(\chi) = \int_{\widehat{G}^2} e_{-g}(\chi \overline{\psi}) \ d \sigma(\chi) d \sigma(\psi) 
    \\ = \int_{\widehat{G}^2} \overline{\chi}(g) \psi(g) \ d \sigma(\chi) d \sigma(\psi) = |\hat{\sigma}(g)|^2.
\end{multline*}
Thus $m(|\hat{\sigma}|^2) = 0$, as was to be shown.
\end{proof}

\begin{lemma}\label[lemma]{lem:APconvolution}
If $m$ is a mean which massively accumulates at $0$ in $bG$ and $U\subseteq G$ is a Bohr set, then $\{g \in G: m*1_U(g) >0\}$ contains a Bohr set with frequencies belonging to $\spec(m)$.
\end{lemma}

\begin{proof}  

Let $f:G\to [0,1]$ be a uniformly almost periodic function supported on $U$, with $f(0)=1$. Let $V\subseteq G$ be a Bohr set with $V\subseteq \{g \in G: f(g) > 1/2\}$.  Since $m$ massively accumulates at $0$, we have $m*1_V(0)=m(1_V)>0$.  It follows that $m*f(0)\geq \frac{1}{2}m*1_V(0)>0$.

Let $\epsilon = m*f(0)/2 > 0$. Since $f$ is uniformly almost periodic, there exist $\chi_1, \ldots, \chi_k \in \widehat{G}$ and $c_1, \ldots, c_k \in \C$ such that 
\[
    \Big| f(g) - \sum_{j=1}^k c_j \chi_j(g) \Big| < \epsilon
\]
for all $g \in G$. Observe that for any $\chi \in \widehat{G}$, we have 
\begin{equation*}
    (m*\chi)(g)=m(\tau_g \chi)=m(\chi(-g)\chi)= m(\chi) \overline{\chi}(g).
\end{equation*}
Therefore,
\begin{multline*}
    \epsilon > \Big| m*f(g) - m* \left( \sum_{j=1}^k c_j \chi_j \right)(g) \Big| = \Big| m*f(g) - \sum_{j=1}^k c_j m(\chi_j) \overline{\chi_j}(g)\Big| \\
    = \Big| m*f(g) - \sum_{\substack{1 \leq j \leq k \\ m(\chi_j) \neq 0}} c_j m(\chi_j) \overline{\chi_j}(g) \Big|
\end{multline*}
for all $g \in G$. Now $\{g \in G: m*f(g) > 0\}$ contains the set 
\[
    \{g \in G: |m*f(g) - m*f(0)| < \epsilon\}
\]
which in turn contains the set
\[
    B = \{g \in G: |\chi_j(g) - 1| < \eta \text{ for } 1 \leq j \leq k, m(\chi_j) \neq 0\}
\]
where 
\[
    \eta = \frac{\epsilon}{\sum_{\substack{1 \leq j \leq k \\ m(\chi_j) \neq 0}} |c_j m(\chi_j)|}.
\]
Since $B$ is a Bohr set whose frequencies belong to the spectrum of $m$, we are done.
\end{proof}

\subsection{Proof of \texorpdfstring{\cref{th:SpectralSumDifference_intro}}{Theorem A}}

We are now ready to prove \cref{th:SpectralSumDifference_intro}. In fact, we prove a stronger version of \cref{th:SpectralSumDifference_intro}, namely \cref{th:SpectralSumDifference}, which not only shows the existence of Bohr sets in $A - A + S$ but also provides information on the spectrum of Bohr sets that appear. The proof of \cref{th:SpectralSumDifference} draws inspiration from Kamae and Mend{\`e}s France's proof \cite{Kamae_France78} of the Furstenberg-S\'ark\"ozy theorem \cite{Furstenberg77, Sarkozy78}, generalized to the setting of discrete abelian groups. 


We say a mean $m$ has \emph{rational spectrum} if for all $\chi \in \spec(m)$, $\ker(\chi) := \{g \in G: \chi(g) = 1\}$ has finite index in $G$. The reason for this terminology is because in case $G = \Z$, $m$ has rational spectrum if and only if $\spec(m)$ consists of characters of the form $n \mapsto e(rn)$ where $e(t)$ denotes $\exp(2\pi i t)$ and $r \in [0, 1)$ is rational. Furthermore, if $\spec(m)$ consists only of the trivial character $\chi_0$, then we say that $m$ has \emph{trivial spectrum}.

\begin{theorem}
\label{th:SpectralSumDifference}  
Let $S \subseteq G$ be a set supporting a $\KMF$ mean $m$. Let $\sigma$ be a noncontinuous positive Radon measure on $\widehat{G}$. 
Then $S - \{g \in G: \Re \hat{\sigma}(g) > 0\}$  contains a Bohr set with frequencies belonging to $\spec(m)$.

In particular, $S - \{g \in G: \Re \hat{\sigma}(g) > 0\}$ contains a finite index subgroup of $G$ if $m$ has rational spectrum and is equal to $G$ if $m$ has trivial spectrum.

\end{theorem}

\begin{proof}
Let $m$, $S$, and $\sigma$ be as in the hypothesis. By \cref{lem:ConvolutionSupport}, we have for each $U\subseteq G$,
\[
    \{g \in G: (m*\Re(1_U\hat{\sigma}))(g)>0\} \subseteq S - \{g \in G: \Re (1_U\hat{\sigma})(g) > 0\} \subseteq S - \{g \in G: \Re \hat{\sigma}(g) > 0\}.
\]
Therefore it suffices to prove that for some $U\subseteq G$, the set $\{g \in G: (m* \Re (1_U\hat{\sigma})) (g) >0\}$ contains a Bohr set with frequencies belonging to $\spec(m)$.

Write $\sigma$ as $\sigma_d+\sigma_c$, where $\sigma_d$ is discrete and $\sigma_c$ is continuous.  For $g \in G$, we have
\[
    \hat{\sigma}_d(g) = \int_{\widehat{G}} \overline{\chi}(g) \  d \sigma_d(\chi) = \sum_{\chi \in \widehat{G}} \sigma(\{\chi\}) \overline{\chi}(g).
\]
Since $\sigma$ is not continuous, $\sigma_d$ is not zero, and we get that
\[
\hat{\sigma}_d(0)=\sum_{\chi\in\widehat{G}} \sigma(\{\chi\})>0.
\]
Now $\hat{\sigma}_d$ is uniformly almost periodic and $\hat{\sigma}_d(0)>0$.  By Proposition \ref{lem:BohrFourierEquivalent}, we can fix $\beta>0$ and a Bohr set $U\subseteq \{g \in G:\Re \hat{\sigma}_{d}(g) > \beta\}$.   Let $\phi=1_U\hat{\sigma}$, so that $\phi=1_U\hat{\sigma}_c+1_U\hat{\sigma}_d$.  

We will show that $m* \phi = m* \Re (1_U\hat{\sigma}_d)$. It suffices to show $m* |1_U\hat{\sigma}_c| = 0$. Since $m*|1_U\hat{\sigma}_c|\leq m*|\hat{\sigma}_c|$ pointwise, it suffices to show that the latter is identically $0$. Now for every $g \in G$,
\[
    m*(|\hat{\sigma}_c|)(g) = m(|\tau_g \hat{\sigma}_c|).
\]
For $h \in G$, we have
\[
    (\tau_g \hat{\sigma}_c)(h) = \hat{\sigma}_c(h - g) = \int_{\widehat{G}} \overline{\chi}(h - g) \ d \sigma_c(\chi) = \int_{\widehat{G}} \overline{\chi}(h) e_g(\chi) \ d \sigma_c(\chi) = \hat{\nu}_g(h) 
\]
where $\nu$ is a measure satisfying $d \nu_g = e_g d \sigma_c$. Since $\sigma_c$ is a continuous measure, $\nu_g$ is also continuous, and can be written as $\nu_1^{+}-\nu_1^{-} +i\nu_{2}^{+}-i\nu_{2}^{-}$, where $\nu_{i}^{\pm}$ are each continuous positive Radon measures. Because $m$ annihilates continuous measures, we deduce that for all $g \in G$
\[
    m(|\tau_g \hat{\sigma}_c|) = m(|\hat{\nu}_g|) = 0,
\]
and so $m*|\hat{\sigma}_c(g)| = 0$, as was to be shown.

\cref{lem:APconvolution} implies $\{g \in G: m* 1_U (g) > 0\}$ contains a Bohr set with frequencies belonging to $\spec(m)$. Since $\Re (1_U\hat{\sigma}_d) \geq  \beta 1_U$ pointwise, we get $\{g \in G: m* 1_U(g) > 0\} \subseteq \{g\in G: m*(\Re 1_U\hat{\sigma}_d)(g)>0\}$. The first assertion of our theorem then follows.

To prove the second assertion, assume $m$ has rational spectrum. From above, $S - \{g \in G: \Re \hat{\sigma}(g) > 0\}$ contains the Bohr set
\[
    B = \{g \in G: |\chi_j(g) - 1| < \eta \text{ for } 1 \leq j \leq k\}
\]
for some $k \in \N, \eta > 0$, and $\chi_1, \ldots, \chi_k \in \spec(m)$. Note that
\[
    B \supseteq \{g \in G: \chi_j(g) = 1 \text{ for } 1 \leq j \leq k\} = \bigcap_{j=1}^k \ker(\chi_j),
\]
which is a subgroup of finite index of $G$ since each $\ker(\chi_j)$ is. Now if $m$ has trivial spectrum, then $S - \{g \in G: \Re \hat{\sigma}(g) > 0\}$ contains
\[
    B = \{g \in G: |\chi_0(g) - 1| < \eta\}
\]
for some $\eta > 0$. Since $\chi_0(g) = 1$ for all $g$, the right hand side above is $G$. 
\end{proof}

\begin{proof}[Proof of \cref{th:SpectralSumDifference_intro}]
Let $S \subseteq G$ be a set that supports a $\KMF$ mean and let $A \subseteq G$ with $d^*(A) > 0$.

Let $m$ be an invariant mean on $\ell^{\infty}(G)$ such that $m(1_A) > 0$. As mentioned in \cref{sec:positive_definite}, the correlation sequence $\phi_{1_A}(g) = m(1_A \cdot \tau_g 1_A)$ is positive definite. By the Bochner-Herglotz theorem (\cref{th:B-H}), there exists a positive finite Radon measure $\sigma$ on $\widehat{G}$ such that $\phi_{1_A}(g) = \hat{\sigma}(g)$ for all $g \in G$. 

Applying \cref{lem:HilbertianInequality} to the case $V = \ell^\infty(G), \langle f_1, f_2 \rangle = m(f_1 \overline{f_2})$, $v = 1_G$, and $w = 1_A$, we obtain
\[
    m(\hat{\sigma}) = m(\phi_{1_A}) \geq m(1_A)^2 > 0.
\]
By \cref{lem:meanOfSigmaHat}, $\sigma(\{\chi_0\}) = m(\hat{\sigma})$ where $\chi_0$ is the trivial character of $G$. Therefore, $\sigma(\{\chi_0\}) > 0$.

We claim that 
\[
    \{g \in G: \Re \hat{\sigma}(g) > 0\} \subseteq A - A.
\]
Indeed, if $\Re \hat{\sigma}(g) > 0$, then 
\[
    m(1_A \cdot \tau_g 1_A) = \phi_{1_A}(g) = \hat{\sigma}(g) > 0.
\]
Therefore, $1_A(h) \cdot 1_A(h - g) > 0$ for some $h \in G$, meaning $h, h - g \in A$ and so $g = h - (h - g) \in A - A$.

Thus,
\[
    S + A - A = S - (A - A) \supseteq S - \{g \in G: \Re \hat{\sigma}(g) > 0\},
\]
which contains a Bohr set according to \cref{th:SpectralSumDifference}, as was to be shown.
\end{proof}

\subsection{Other properties of sets that support \texorpdfstring{$\KMF$}{KMF} means}

In this section, we show that the property of supporting an $\KMF$ mean is partition regular, and that any set supporting a $\KMF$ mean is both a van der Corput set and a set of nice recurrence. We start with a lemma.

\begin{lemma}\label[lemma]{prop:RelativeAnnihilates}
If $m$ is a mean on $\ell^{\infty}(G)$ which annihilates continuous measures and $A\subseteq G$ has $m(1_A)>0$, then the mean $\nu$ given by $\nu(f):=\frac{1}{m(1_A)} m(f1_A)$ also annihilates continuous measures.
\end{lemma}
\begin{proof}
If $m$ annihilates continuous measures, $\sigma$ is a continuous positive Radon measure on $\widehat{G}$, and $m(A)>0$, then 
\[
    \frac{1}{m(1_A)}m(1_A|\hat{\sigma}|^2)\leq \frac{1}{m(1_A)}m(|\hat{\sigma}|^2)=0. \qedhere
\]
\end{proof}

\begin{proposition}\label[proposition]{prop:partition_regular_KMF_mean}
If $S = S_1 \cup S_2 \subseteq G$ supports $\KMF$ mean, then $S_1$ or $S_2$ also supports a $\KMF$ mean. 
\end{proposition}
\begin{proof}
Suppose $S = S_1 \cup S_2$ supports a $\KMF$ mean $m$.
Without loss of generality, we may assume $S_1\cap S_2=\varnothing$.  If $m(S_i)=1$ for some $i$, then $m$ is supported on $S_i$, and we are done.  Otherwise,  for $i=1,2$, define the means $m_i$ by $m_i(f):=\frac{1}{m(S_i)}m(1_{S_i}f)$ for $f \in \ell^{\infty}(G)$.  Then $m_i$ is supported on $S_i$ and by \cref{prop:RelativeAnnihilates} both $m_1$ and $m_2$ annihilate continuous measures. Furthermore $m = c_1m_1+c_2m_2$, where $c_i=m(S_i)$.

We claim that one of the $m_i$ massively accumulates at $0$ in $bG$.  Assume, to get a contradiction, that neither $m_1$ nor $m_2$ massively accumulates at $0$ in $bG$.  Then there are Bohr sets $U_1$, $U_2$ such that $m_1(1_{U_1})=0$ and $m_2(1_{U_2})=0$. Setting $V=U_1\cap U_2$, we get that $V$ is a Bohr set with $m_1(1_V)=m_2(1_V)=0$.  Then $m(1_V)=c_1m_1(1_V)+c_2m_2(1_V)=0$, as well.  Thus $m$ does not massively accumulate at $0$ in $bG$.
\end{proof}

The next lemma provides a general criterion for proving that a set is both a van der Corput set and a set of nice recurrence. We will use it to show that any set supporting a $\KMF$ mean enjoys both properties. The lemma will be used again in the context of $\AlmostBohrBohr$ sets in \cref{sec:abb_vanderCorput_nice_recurrence}.

\begin{lemma}\label[lemma]{lem:very_general_way_to_get_nice_vdC}
Let $S \subseteq G$. If for every positive finite Radon measure $\sigma$ on $\widehat{G}$  and every $\epsilon > 0$, there exists $h \in S$ such that
\[
    \Re (\hat{\sigma}(h)) > \sigma(\{\chi_0\}) - \epsilon,
\]
then $S$ is a van der Corput set and a set of nice recurrence.
\end{lemma}
\begin{proof}   
Let $S \subseteq G$ be a set that satisfies the lemma's hypothesis. To show that $S$ is a van der Corput set, we will prove that for every positive, finite Radon measure on $\widehat{G}$, if $\sigma(\{\chi_0\}) > 0$, then there exists $h \in S$ such that $\hat{\sigma}(h) \neq 0$. 
However, this is now obvious since by the hypothesis, there exists $ h \in S$ for which 
\[
    \Re(\hat{\sigma}(h)) > \sigma(\{\chi_0\}) - \frac{\sigma(\{\chi_0\})}{2} > 0.
\]

To show that $S$ is a set of nice recurrence, let $(X, \mu, (T_g)_{g \in G})$ be a $G$-measure preserving system, and $A \subseteq G$ be a measurable set with $\mu(A) > 0$.
For $g \in G$, define $c_g = \mu(A \cap T_g A)$. Then the function $g \mapsto c_g$ is positive definite. By the Bochner-Herglotz theorem (\cref{th:B-H}), there exists a nonnegative Radon measure $\sigma$ on $\widehat{G}$ such that $c = \hat{\sigma}$, i.e. 
\[
    c_g = \int_{\widehat{G}} \overline{\chi}(g) \, d \sigma(\chi) \text{ for every $g \in G$.}
\]
By Lemmas \ref{lem:HilbertianInequality} and \ref{lem:meanOfSigmaHat}, we have $\sigma( \{\chi_0\}) \geq c_0^2 = \mu(A)^2$. According to the hypothesis, for any $\epsilon > 0$, there exists $h \in S$ such that
\[
    \mu(A \cap T_h A) = \Re(\hat{\sigma}(h)) > \sigma(\{\chi_0\}) - \epsilon \geq \mu(A)^2 - \epsilon.
\]
\end{proof}

\begin{proposition}\label[proposition]{prop:kmf_imply_vanderCorput}
If $S \subseteq G$ supports a $\KMF$ mean, then $S$ is a van der Corput set and is a set of nice recurrence.
\end{proposition}

\begin{proof}
Let $S \subseteq G$ be a set that supports a $\KMF$ mean $m$ and let $\sigma$ be a positive finite Radon measure on $\widehat{G}$. Without loss of generality, assume that $\sigma(\widehat{G}) =1$. Let $\epsilon > 0$ be arbitrary. In view of \cref{lem:very_general_way_to_get_nice_vdC}, it suffices to show there exists $h \in S$ such that
\[
     \Re (\hat{\sigma}(h)) > \sigma(\{\chi_0\}) - \epsilon.
\]

Decompose $\sigma = \sigma_c + \sigma_d$ where $\sigma_c$ is continuous and $\sigma_d$ is discrete. We abuse notation and write $\hat{\sigma}_d$ and $\hat{\sigma}_c$ in place of $\widehat{\sigma_d}$ and $\widehat{\sigma_c}$. Note that the number of atoms of $\sigma_d$ is at most countable and for every $g \in G$,
\[
\hat{\sigma}_d(g) = \sum_{\chi \in \widehat{G}} \overline{\chi}(g) \sigma_d( \{ \chi \}). 
\] 
In particular, 
\[
    \hat{\sigma}_d(0) = \sum_{\chi \in \widehat{G}} \sigma_d( \{ \chi \})  \geq \sigma_d (\{ \chi_0 \}) = \sigma (\{ \chi_0 \}).
\]
Let 
\[
    B =  \left\{ g \in G: \Re \left(\hat{\sigma}_d(g)\right) > \sigma(\{\chi_0\}) - \frac{\epsilon}{2}\right\}.
\]
Then $B$ contains 
\[
    \left\{ g \in G: \left| \hat{\sigma}_d(g) - \hat{\sigma}_d(0) \right| < \frac{\epsilon}{2} \right\}.
\]
Since $\sigma_d(\widehat{G}) \leq 1$, there exists a finite set $\Lambda \subseteq \widehat{G}$ such that $\sum_{\chi \in \widehat{G} \setminus \Lambda} \left| \sigma_d( \{ \chi \}) \right| < \frac{\epsilon}{8}$. 
Note that
\[
    \hat{\sigma}_d(g) - \hat{\sigma}_d(0) = \sum_{\chi \in \Lambda} (\overline{\chi}(g) - 1) \sigma_d(\{\chi\}) + \sum_{\chi \in \widehat{G} \setminus \Lambda} (\overline{\chi}(g) - 1) \sigma_d(\{\chi\}).
\]
Therefore, $B$ contains the Bohr set 
\[
    \Big\{ g \in G: \sum_{\chi \in \Lambda} \left| \overline{\chi}(g) - 1 \right| \sigma_d( \{ \chi \} )\Big\} < \frac{\epsilon}{4}.
\]
Since $m$ massively accumulates at $0$ in $bG$, we get $m(1_B) > 0$. 

Let $E: = \left\{ g \in G: \left| \hat{\sigma}_c(g) \right| \geq \frac{\epsilon}{2}\right\}$.
Then $|\hat{\sigma}_c(g)| \geq \frac{\epsilon}{2} 1_E(g)$ for all $g \in G$. Since $m$ annihilates continuous measures, we get
\[
    m(1_E) \leq \frac{2 m(|\hat{\sigma}_c|)}{\epsilon} = 0.
\]
It follows that
\[
    m(1_{B \setminus E}) \geq m(1_B) - m(1_E) > 0,
\]
and so $S \cap (B \setminus E) \neq \varnothing$. Let $h$ be an element in this nonempty intersection. We have
$\Re(\hat{\sigma}_d(h)) > \sigma(\{\chi_0\}) - \epsilon/2$ and $\Re(\hat{\sigma}_c(h)) > - \epsilon/2$. Therefore \[
   \Re(\hat{\sigma}(h)) = \Re(\hat{\sigma}_d(h)) + \Re( \hat{\sigma}_c(h)) > \sigma(\{\chi_0\}) - \frac{\epsilon}{2} - \frac{\epsilon}{2} = \sigma(\{\chi_0\}) - \epsilon,
\]    
as was to be shown.
\end{proof}

\section{Applications}
\label{sec:means_from_sequences}

\subsection{Bulinski and Fish's theorem and generalizations}

\label{sec:bulinski_fish}

While one can deduce Theorem \ref{cor:Specific-intro_main} from Theorem \ref{th:SpectralSumDifference} relatively quickly, we will prove significantly stronger results than Theorem \ref{cor:Specific-intro_main}, namely Theorems \ref{th:improved_bulinski_fish_polynomial}, \ref{th:improved_bulinski_fish_polynomial_prime} and \ref{thm:bulinski_fish_hardy_field} below. Our motivation for these theorems is the following result of Bulinski and Fish \cite{BulinskiFish_QuantitativeTwisted}.

\begin{theorem}[{\cite[Theorem 1.4]{BulinskiFish_QuantitativeTwisted}}] \label{th:bulinski-fish}
Let $\epsilon > 0$ and assume that $P(n) \in \Z[n]$ is a polynomial such that $\deg(P) \geq 2$ and $P(0) = 0$. Let $E \subseteq \Z^2$ with $d^*(E) > \epsilon$. Then there exists $k \leq k (\epsilon, P)$ such that
\[
    k \Z \subseteq \{x + P(y): (x, y) \in E - E\}.
\]
\end{theorem}

\begin{remark}
It is necessary in \cref{th:bulinski-fish} that the degree of the polynomial $P$ is at least $2$. In fact, if $E \subseteq \Z^2, d^*(E) >0$, then $\{ x+y : (x,y) \in E-E\} $ does not necessarily contain a Bohr set. Indeed, by Kriz's example \cite{Kriz87}, there exists $A \subseteq \Z$ with $d^*(A) > 0$ such that $A-A$ does not contain a Bohr set. Let $E=\{(x,y): x+y\in A\}$.  Then it is easy to show $d^*_{\Z^2}(E) > 0$. Now for $(x', y') \in E - E$, there exist $(x_1, y_1), (x_2, y_2) \in E$ such that
\[
    (x', y') = (x_1, y_1) - (x_2, y_2).
\]
Therefore,
\[
    x' + y' = (x_1 - x_2) + (y_1 - y_2) = (x_1 + y_1) - (x_2 + y_2) \in A - A.
\]
Thus the set $\{x' + y': (x', y') \in E - E\}$ is a subset of $A - A$ and so it does not contain a Bohr set.
\end{remark}

By taking $E = A \times B$ where $A, B \subseteq \Z$ have positive upper Banach density, \cref{th:bulinski-fish} gives us the following corollary in the flavor of $\DiffBohr$ sets.

\begin{corollary}\label[corollary]{cor:bulinski_fish_cor}
For every $A, B \subseteq \Z$ with $d^*(A)$, $d^*(B) > 0$, and every $P \in \Z[x]$ with $\deg(P) \geq 2$ and $P(0) = 0$,  $A - A + P(B - B)$ contains a finite index subgroup of $\Z$.
\end{corollary}



Using \cref{th:SpectralSumDifference}, we will prove several generalizations of \cref{cor:bulinski_fish_cor} for polynomials not necessarily vanishing at 0, polynomials evaluated at primes and for Hardy field sequences. First we recall the definition of intersective polynomials.

\begin{definition} \label{def:intersective}
A polynomial $h \in \Z[x]$ is said to be \textit{intersective} if for any positive integer $m$, there exists $n \in \Z$ such that $h(n)$ is divisible by $m$. The polynomial $h$ is said to be \textit{intersective of the second kind} if for any positive integer $m$, there exists $n \in \Z$ such that $h(n)$ is divisible by $m$ and $\gcd(m,n)=1$.
\end{definition}

It is well-known that $h$ is intersective if and only if the set $\{ h(n): n \in \Z \}$ is a set of recurrence for $\Z$-measure preserving systems, and $h$ is intersective of the second kind if and only if the set $\{ h(p): p \textup{ prime} \}$ is a set of recurrence for $\Z$-measure preserving systems (see \cite{LeThaiHoang_survey}). The notion of intersectivity was generalized to a family of polynomials in \cite{Bergelson_Leibman_Lesigne08}.
\begin{definition}
A family of polynomials $Q_1, \ldots, Q_d \in \Z[x]$ is said to be \textit{jointly intersective} if for any positive integer $m$, there exists $n \in \Z$ such that $Q_1(n), \ldots, Q_d(n)$ are all divisible by $m$. The family is said to be \textit{jointly intersective of the second kind} if for any positive integer $m$, there exists $n \in \Z$ such that $\gcd(m,n)=1$ and $Q_1(n), \ldots, Q_d(n)$ are all divisible by $m$. 
\end{definition}

Similarly to the notion of intersective polynomials, polynomials $Q_1, \ldots, Q_d$ are jointly intersective if and only if the set $\{ (Q_1(n), \ldots, Q_d(n)): n \in \Z \}$ is a set of recurrence for $\Z^d$-measure preserving systems (see \cite{Bergelson_Leibman_Lesigne08}). The polynomials $Q_1, \ldots, Q_d$ are jointly intersective of the second kind if and only if the set $\{ (Q_1(p), \ldots, Q_d(p)): p \textup{ prime} \}$ is a set of recurrence for $\Z^d$-measure preserving systems and this fact follows from \cref{th:improved_bulinski_fish_polynomial_prime} below. If all the $Q_i$'s have a common integer root, then they are jointly intersective, but there are other examples of jointly intersective polynomials. Likewise, if all the $Q_i$'s have $1$ or $-1$ as a common root, then they are jointly intersective of the second kind, but these are not all the examples. We refer the reader to \cite[Section 2]{Le_Spencer_Intersective_1}, \cite[Section 2]{Le_Spencer_Intersective_2} for properties and examples of (jointly) intersective polynomials.


\begin{theorem} \label{th:improved_bulinski_fish_polynomial}
Let $P, Q_1, \ldots, Q_d$ be jointly intersective polynomials, $A \subseteq \Z$ and $B \subseteq \Z^d$ with $d_{\Z}^*(A)>0$ and $d_{\Z^d}^*(B) > 0$. Let
\[
 S = \{ P(n) : n \in \N, (Q_1(n), \ldots, Q_d(n)) \in B-B\}.
\]
Then $A-A+S$ contains a Bohr set. Furthermore, if $P$ is not in the $\Q$-vector space spanned by $Q_1, \ldots, Q_d$, then $A-A+S$ contains a subgroup of finite index of $\Z$.
 \end{theorem}

\begin{remark}
Taking $P(0)=0, \deg P \geq 2$, $d=1$ and $Q_1(n) = n$ in Theorem \ref{th:improved_bulinski_fish_polynomial}, we recover \cref{cor:bulinski_fish_cor}.  

Taking $d = 1, P(n) = Q_1(n) = n$ and $A = B$, we recover Bogolyubov's theorem (\cref{th:bogo}). Note that the Bohr set in \cref{th:bogo} cannot be upgraded to a subgroup of finite index in general. This fact demonstrates the necessity of the linear independence condition in the second part of Theorem \ref{th:improved_bulinski_fish_polynomial}. 
\end{remark}

Before embarking on the proof of Theorem \ref{th:improved_bulinski_fish_polynomial}, let us state some basic properties of jointly intersective polynomials. Suppose now $Q_1, \ldots, Q_d \in \Z[x]$ are jointly intersective. The condition of joint intersectivity is equivalent to the property that for each prime $p$, $Q_1, Q_2, \ldots, Q_d$ have a common $p$-adic integer root $z_p$. 
By the Chinese Remainder Theorem and multiplicativity, for every $m \in \N$ 
we can fix an integer $
r_m \in (-m,0]$ such that $Q_i(r_m) \equiv 0 \pmod{m}$ for $i=1, \ldots, d$. Moreover, $r_{mq} \equiv r_m \pmod{m}$ for all $q \in \N$. If $Q_1, \ldots, Q_d$ are jointly intersective of the second kind, then we can further require $(r_m, m)=1$.

\begin{lemma}\label{lem:basic_intersective}\
    \begin{enumerate}
      \item Suppose $Q_1, \ldots, Q_d \in \Z[x]$ are jointly intersective. Then for any $m$, the polynomials 
      $\frac{1}{m} Q_i (mn + r_{m} )$ ($1 \leq i \leq d$) are also jointly intersective.
      \item Suppose $Q_1, \ldots, Q_d \in \Z[x]$ are jointly intersective and linearly independent over $\Q$. Then $Q_1 - Q_1(0), \ldots, Q_d - Q_d(0) $ are also linearly independent over $\Q$. 
  \end{enumerate}
\end{lemma}

\begin{proof}
    For (i), observe that for any $\ell \in \N$, $n= \frac{r_{m\ell} - r_m}{m}$ satisfies $\frac{1}{m} Q_i (mn + r_{m} ) = \frac{1}{m} Q_i(r_{m \ell}) \equiv 0 \pmod{\ell}$.

For (ii), suppose for a contradiction that $Q_1-Q_1(0), \ldots, Q_\ell - Q_\ell(0)$ are not $\Q$-linearly independent. Then there exist $v_1, \ldots, v_d \in \Z$ such that $v_1Q_1 + \cdots + v_\ell Q_\ell$ is a nonzero constant. This is impossible since $Q_1, \ldots, Q_\ell$ are jointly intersective.
\end{proof}


\begin{proof}[Proof of Theorem \ref{th:improved_bulinski_fish_polynomial}] By the discussion preceding Lemma \ref{lem:basic_intersective}, for each $m \in \N$, we can fix an integer $r_m \in (-m,0]$ such that for every $P(r_m), Q_1(r_m), \ldots, Q_d(r_m)$ are all divisible by $m$. Moreover, $r_{mq} \equiv r_m \pmod{m}$ for all $q \in \N$. We write $\vec{Q} = (Q_1, \ldots, Q_d)$.

For simplicity, we will prove Theorem \ref{th:improved_bulinski_fish_polynomial} in the case where $Q_1, \ldots, Q_d$ are linearly independent over $\Q$. At the end of the proof, we will indicate the necessary changes in the general case. 



Let $\nu$ be an invariant mean on $\Z^d$ such that $\nu(B) = d^*(1_B) >0$. By the Bochner-Herglotz theorem (Theorem \ref{th:B-H}), there exists a nonnegative Radon measure $\sigma$ on $\T^d$ such that
\begin{equation} \label{eq:sigma}
\nu(1_B \cdot 1_{B+\vec{m}}) = \int_{\T^d} e(\vec{\beta} \cdot \vec{m}) \, d \sigma(\vec{\beta})
\end{equation}
for every $\vec{m} \in \Z^d$. 


For $\alpha \in \T$, we consider
\begin{equation} \label{eq:s-alpha}
S_N (\alpha) = \frac{1}{N} \sum_{n=1}^N e(P(n) \alpha) \cdot \nu(1_B \cdot 1_{B+\vec{Q}(n)}).
\end{equation}

\noindent \textbf{Claim 1:} For all but countably many $\alpha \in \T$, we have $\lim_{N \rightarrow \infty} S_N(\alpha) =0 $. Furthermore, if $P$ is not in the $\Q$-vector space spanned by $Q_1, \ldots, Q_d$, then the exceptional $\alpha$'s are all rational.
\begin{proof}[Proof of Claim 1]
We have
\begin{eqnarray*}
S_N (\alpha) &=& \frac{1}{N} \sum_{n=1}^N e(\alpha P(n)) \int_{\T^{d}} e(\vec{\beta} \cdot \vec{Q}(n)) \, d \sigma(\vec{\beta})\\
&=& \int_{\T^{d}} \left( \frac{1}{N} \sum_{n=1}^N e \left(\alpha P(n) + \vec{\beta} \cdot \vec{Q}(n) \right) \right)\, d \sigma(\vec{\beta}).
\end{eqnarray*}


First we assume that $P, Q_1, \ldots, Q_d$ are $\Q$-linearly independent. By Lemma \ref{lem:basic_intersective}, $P-P(0), Q_1-Q_1(0), \ldots, Q_d-Q_d(0)$ are also $\Q$-linearly independent.
It follows that whenever at least one of $\alpha, \gamma_1, \ldots, \gamma_d$ is irrational,  the polynomial $\alpha P + \gamma Q_1 + \cdots +\gamma_\ell Q_\ell$ must have an irrational, non-constant coefficient. Consequently, whenever $\alpha$ is irrational, the polynomial $\alpha P +  \vec{\beta} \cdot \vec{Q}$ must have an irrational, non-constant coefficient.
In this case, Weyl's theorem implies that
\[
\lim_{N \rightarrow \infty} \frac{1}{N} \sum_{n=1}^N e \left(P(n) \alpha + \vec{Q}(n) \cdot \vec{\beta} \right) =0.
\]
By the dominated convergence theorem, we conclude that $\lim_{N \rightarrow \infty} S_N(\alpha) = 0$ whenever $\alpha$ is irrational.

Suppose now that $P = r_1 Q_1 + \cdots + r_\ell Q_d$ for some $\vec{r} = (r_1, \ldots, r_d) \in \Q^d \setminus \{ \vec{0}\}$. Then 
\[
\lim_{N \rightarrow \infty} \frac{1}{N} \sum_{n=1}^N e \left( \alpha P(n) + \vec{\beta}\cdot \vec{Q}(n) \right) = 
\lim_{N \rightarrow \infty} \frac{1}{N} \sum_{n=1}^N e \left( (\alpha \vec{r} + \vec{\beta}) \cdot \vec{Q}(n) \right),
\]
By Lemma \ref{lem:basic_intersective}, $Q_1-Q_1(0), \ldots, Q_d-Q_d(0)$ are $\Q$-linearly independent. By Weyl's theorem, the latter limit is 0 unless $\alpha \vec{r} + \vec{\beta} \in \Q^d$. Letting $X = \{ \vec{\beta} \in \T^d: \sigma(\{ \vec{\beta} \}) >0 \}$ be the set of atoms of $\sigma$, then $X$ is countable. For each $\vec{\beta} \in X$, there are at most countably many $\alpha \in \T$ such that $\alpha \vec{r} + \vec{\beta} \in \Q^d$ for some $\vec{r} \in \Q^d \setminus \{ \vec{0} \}$.
%
Consequently, for all but countably many $\alpha$, $\lim_{N \rightarrow \infty} S_N(\alpha) =0$,
by the dominated convergence theorem. Hence, the claim is proved. 
\end{proof}

\noindent \textbf{Claim 2:}
Let $C$ be any subset of $\Z$ with $d^*(C) >0$. We have
\[
\liminf_{N \to \infty} \frac{1}{N} \sum_{n=1}^N d^*(C \cap (C+P(n))) \cdot \nu(1_B \cdot 1_{B+\vec{Q}(n)}) >0.
\]


\begin{proof}[Proof of Claim 2] This is essentially a proof that the set $\{ (P(n), Q_1(n), \ldots, Q_d(n)): n \in \Z \}$ is a set of measurable recurrence in $\Z^{d+1}$, by adapting Furstenberg's proof for $\Z$ in \cite{Furstenberg77}. Let $\lambda$ be an invariant mean on $\Z$ such that $\lambda(1_C) = d^*(C) >0$. Then there exists a nonnegative Radon measure $\tau$ on $\T$ such that
\[
\lambda(1_C \cdot 1_{C+n}) = \int_\T e(n \alpha) \, d \tau(\alpha)
\]
for every $n \in \Z$.

We will prove that 
\[
    \liminf_{N \to \infty} \frac{1}{N} \sum_{n=1}^N \lambda(1_C \cdot 1_{C+P(n)}) \cdot \nu(1_B \cdot 1_{B+\vec{Q}(n)}) > 0.
\]
Note that
\begin{equation*}
 \lambda(1_C \cdot 1_{C+P(n)}) \cdot \nu(1_B \cdot 1_{B+\vec{Q}(n)}) = \int_{\T^{d+1}} \left(  e( \alpha P(n) + \vec{\beta} \cdot \vec{Q}(n) )\right) d (\tau \times \sigma) (\alpha, \vec{\beta}).    
\end{equation*}

We decompose $\tau \times \sigma = (\tau \times \sigma)_c + (\tau \times \sigma)_d$ as a sum of continuous and discrete parts. Note that the atoms of $\tau \times \sigma$ are precisely pairs $(\alpha, \vec{\beta})$ where $\alpha$ is an atom of $\tau$ and $\vec{\beta}$ is an atom of $\sigma$. By Lemma \ref{lem:meanOfSigmaHat}, we have $\tau(\{ 0 \}) \geq \lambda(1_C) = d^*(C)$ and $\sigma(\{ \vec{0} \}) \geq \nu(1_B) = d_{\Z^d}^*(B)$. 

Let $0<\epsilon < d_{\Z^d}^*(B) d^*(C)$. Again, let us distinguish two cases.\\

\noindent \textit{Case 1.} Suppose that $P, Q_1, \ldots, Q_d$ are $\Q$-linearly independent.  Let $M>0$ be chosen later and $X_M$ denote the set of all rational elements in $\T$ with denominator at most $M$. By partitioning the atoms of $(\tau \times \sigma)_d$, we further decompose 
\[
(\tau\times \sigma)_d = (\tau\times \sigma)'_d + (\tau\times \sigma)^{''}_d,
\]
where an atom $(\alpha, \vec{\beta})$ of $(\tau \times \sigma)_d$ is an atom of $(\tau \times \sigma)^{''}_d$ if and only if all components of $(\alpha, \vec{\beta})$ are rational and at least one of them is not in $X_M$. We can choose $M$ so that $(\tau \times \sigma)^{''}_d (\T^{d+1}) < \epsilon$.

 Let $P^*(n) := P(r_{M!} + M! \cdot n)$ and $Q_i^*(n) := Q(r_{M!} + M! \cdot n)$ for $i=1, \ldots, d$, then $P^*(n)$ and $Q_i^*(n)$ are divisible by $M!$ for any $n \in \Z$. We write $\vec{Q^*} = (Q_1^*, \ldots, Q_d^*)$ and consider
\[
T_N := \frac{1}{N} \sum_{n=1}^N \lambda(1_C \cdot 1_{C+P^*(n)}) \cdot \nu(1_B \cdot 1_{B+\vec{Q^*}(n)}).
\]
Then 
\begin{align}
T_N &=  \int_{\T^{d+1}} \frac{1}{N} \sum_{n=1}^N e \left(  \alpha P^*(n) + \vec{\beta} \cdot \vec{Q^*}(n) \right) d (\tau \times \sigma) (\alpha, \vec{\beta}) \nonumber \\
&  = \int_{\T^{d+1}} \frac{1}{N} \sum_{n=1}^N e \left(  \alpha P^*(n) + \vec{\beta} \cdot \vec{Q^*}(n) \right) d (\tau \times \sigma)_c (\alpha, \vec{\beta}) \label{eq:int1} \\
& \qquad + \int_{\T^{d+1}} \frac{1}{N} \sum_{n=1}^N e \left(  \alpha P^*(n) + \vec{\beta} \cdot \vec{Q^*}(n) \right) d (\tau \times \sigma)^{'}_d (\alpha, \vec{\beta}) \label{eq:int2} \\
&  \qquad + \int_{\T^{d+1}} \frac{1}{N} \sum_{n=1}^N e \left(  \alpha P^*(n) + \vec{\beta} \cdot \vec{Q^*}(n) \right) d (\tau \times \sigma)^{''}_d (\alpha, \vec{\beta}). \label{eq:int3}
\end{align}
By Lemma \ref{lem:basic_intersective} (i), $P^*, Q_1^*, \ldots, Q_d^*$ are jointly intersective. By Lemma \ref{lem:basic_intersective} (ii), $P^*(n)-P^*(0), Q_1^*(n)-Q_1^*(0), \ldots, Q_\ell^*-Q_\ell^*(0)$ are linearly independent over $\Q$. Therefore, as we have seen before, 
\[
    \lim\limits_{N \rightarrow \infty} \frac{1}{N} \sum_{n=1}^Ne \left(  \alpha P^*(n) + \vec{\beta} \cdot \vec{Q^*}(n) \right) =0
\]
for all $(\alpha, \beta) \in \R^{d+1} \setminus \Q^{d+1}$.
By the dominated convergence theorem, the limit of the integral in \eqref{eq:int1} is 0 as $N$ goes to infinity. Also, by the dominated convergence theorem, the limit as $N$ goes to infinity of the integral in \eqref{eq:int2} exists, and is at least $(\tau \times \sigma) (\{ 0,\vec{0} \}) \geq d_{\Z^d}^*(B) d^*(C)$. By choice of $M$, the integral in \eqref{eq:int3} is at most $\epsilon$ in absolute value. Therefore, for $N$ sufficiently large, we have $T_N > d_{\Z^d}^*(B)d^*(C) - \epsilon$. Hence, $\liminf_N T_N >0$ and the claim is proved in this case.\\

\noindent \textit{Case 2.} Suppose that $P, Q_1, \ldots, Q_d$ are linearly dependent over $\Q$; that is $P = \vec{r} \cdot \vec{Q}$ for some $\vec{r} \in \Q^d$. This time, we decompose
\[
(\tau\times \sigma)_d = (\tau\times \sigma)'_d + (\tau\times \sigma)^{''}_d,
\]
where an atom $(\alpha, \vec{\beta})$ of $(\tau \times \sigma)_d$ is an atom of $(\tau \times \sigma)^{''}_d$ if and only if $\alpha \vec{r} +  \vec{\beta} \in \Q^d \setminus X_M^d$. We choose $M$ so that $(\sigma\times \tau)^{''}_d (\T^{d+1}) < \epsilon$.

We define $P^*, \vec{Q^*}$ and $T_N$ as in the previous case. Then 
\[
\lim\limits_{N \rightarrow \infty} \frac{1}{N} \sum_{n=1}^N e \left( \alpha P^*(n) + \vec{\gamma} \cdot \vec{Q^*}(n) \right)
= \lim\limits_{N \rightarrow \infty} \frac{1}{N} \sum_{n=1}^N e \left( (\alpha \vec{r} +  \gamma) \cdot \vec{Q^*}(n) \right) =0
\]
whenever $\alpha \vec{r} +  \vec{\gamma} \in \R^{d} \setminus \Q^\ell$. Arguing similarly to the previous case, we again have $\liminf_N T_N > 0$, and the claim is proved. 
\end{proof}
Using Claims 1 and 2, we will now construct a mean that is supported on $S$. 
Consider the means $m_N$ given by 
\[
m_N(f):=\frac{ \sum_{n=1}^N f(P(n)) \cdot \nu(1_B \cdot 1_{B+\vec{Q}(n)})}{\sum_{n=1}^N \nu(1_B \cdot 1_{B+\vec{Q}(n)})},
\]
for every $f \in \ell^\infty(\Z)$. Then $m_N(1_S)=1$. Also, Claim 2 implies that there is a positive constant $c$ such that for sufficiently large $N$,
\[
c N \leq \sum_{n=1}^N \nu(1_B \cdot 1_{B+Q(n)}) \leq N.
\]
Let $m$ be any weak$^*$ cluster point of $m_N$.  Since $\mathcal M(\Z)$, the set of all means on $\Z$, is a compact convex subset of $\ell^\infty(\Z)^*$, $m$ is a mean on $\ell^{\infty}(\Z)$ satisfying $m(1_S)=1$. \\

\noindent \textbf{Claim 3:} $m$ annihilates continuous measures.

\begin{proof}[Proof of Claim 3]
Let $\mu$ be a continuous measure on $\T$. We have
\begin{align*}
m(\hat{\mu}) &= \lim_{N\to\infty} m_N(\hat{\mu})\\
&= \lim_{N\to\infty} \frac{N }{\sum_{n=1}^N \nu(1_B \cdot 1_{B+\vec{Q}(n)})} \int_{\T} S_N(\alpha) \, d\mu(\alpha) ,
\end{align*}
where $S_N(\alpha)$ is defined in \eqref{eq:s-alpha}. By Claim 1, we have $\lim_{N \rightarrow \infty} S_N(\alpha) =0$ for all but countably many $\alpha$. Since $\mu$ is continuous, we have $\lim_{N \rightarrow \infty} \int_{\T} S_N(\alpha) \, d\mu(\alpha) =0$, by the dominated convergence theorem. Hence, $m(\hat{\mu}) = 0$ and the claim follows from Lemma \ref{lem:AnnihilateEquiv}.
\end{proof}

\noindent \textbf{Claim 4:} $m$ massively accumulates at 0 in $b\Z$.    

\begin{proof}[Proof of Claim 4] If $D$ is any Bohr set, then there is a sequence ${N_k}$ such that
\begin{align*}
m(1_D) &= \lim_{k\to\infty} m_{N_k}(1_D)\\
&= \lim_{k\to\infty}\frac{ \sum_{n=1}^{N_k} 1_D(P(n)) \cdot \nu(1_B \cdot 1_{B+\vec{Q}(n)})}{ \sum_{n=1}^{N_k} \nu(1_B \cdot 1_{B+\vec{Q}(n)})}
\end{align*}
The claim now follows from Claim 2, by taking a Bohr set $C$ with $C-C \subseteq D$.    
\end{proof}
From Claims 3 and 4 it follows that $m$ is a KMF mean supported on $S$. By \cref{th:SpectralSumDifference_intro}, we conclude that for any $A \subseteq \Z$ with $d^*(A)>0$, $A-A+S$ contains a Bohr set. 

Finally, suppose now that $P$ is not in the $\Q$-vector space spanned by $Q_1, \ldots, Q_d$. For $\alpha \in \T$, let $\chi_\alpha$ be the character $n \mapsto e(n \alpha)$ on $\Z$. Then
\[
\widehat{m}(\chi_\alpha) = m(\overline{\chi_\alpha}) = \lim_{N \rightarrow \infty} \frac{S_N(-\alpha)}{\sum_{n=1}^{N} \nu(1_B \cdot 1_{B+\vec{Q}(n)})},  
\]
where $S_N(\alpha)$ is defined in \eqref{eq:s-alpha}. By Claim 1, we have $\widehat{m}(\chi_\alpha) = 0$ whenever $\alpha$ is irrational. Thus $m$ has rational spectrum and so by the second part of Theorem \ref{th:SpectralSumDifference}, we conclude that $A-A+S$ contains a subgroup of finite index of $\Z$. 

We will now address the general case where $Q_1, \ldots, Q_d$ are not necessarily $\Q$-linearly independent. Without loss of generality, we may assume that $\{ Q_1, \ldots, Q_\ell \}$ is a basis of the $\Q$-vector space spanned by $Q_1, \ldots, Q_d$. Let $\vec{Q'} = (Q_1, \ldots, Q_\ell)$. Then there is a linear map $\Phi: \R^d \rightarrow \R^\ell$ with rational entries such that 
\[
\vec{\beta} \cdot \vec{Q} = \Phi(\vec{\beta}) \cdot \vec{Q'}
\]
for every $\vec{\beta} \in \R^d$. 

Let $q \in \N$ be such that $q\Phi$ has integer entries. Thus $\Psi := q\Phi$ can be regarded as a map from $\T^d$ to $\T^\ell$. We will pass to a subset of $S$, namely
\[
S'= \{ P(qn+r_q) : (Q_1(qn+r_q), \ldots, Q_d(qn+r_q)) \in B-B\}.
\]
It suffices to show that $A-A + S'$ contains a Bohr set or a subgroup of finite index.

Let $\sigma$ be the measure defined in \eqref{eq:sigma}, then for every $n \in \Z$, we have
\begin{eqnarray*}
\nu(1_B \cdot 1_{B+\vec{Q}(nq+r_q)}) &=& \int_{\T^d} e \left(\vec{\beta} \cdot \vec{Q}(nq+r_q) \right) \, d \sigma(\vec{\beta})  \\
&=& \int_{\T^d} e \left( \Psi(\vec{\beta}) \cdot \frac{1}{q} \vec{Q'}(nq+r_q) \right) \, d \sigma(\vec{\beta}) \\
&=& \int_{\T^\ell} e \left( \vec{\gamma} \cdot \frac{1}{q} \vec{Q'}(nq+r_q) \right) \, d \rho(\vec{\gamma}),
\end{eqnarray*}
where $\rho = \Psi_*(\sigma)$ is the pushforward of $\sigma$ by $\Psi$, i.e. $\rho(X) = \sigma (\Psi^{-1}(X))$ for any Borel set $X \subseteq \T^\ell$.

We can now repeat the statement and proof of Claim 1 with $S_N(\alpha)$ in \eqref{eq:s-alpha} replaced by
\[
S'_N (\alpha) = \frac{1}{N} \sum_{n=1}^N e(P(qn+r_q) \alpha) \cdot \nu(1_B \cdot 1_{B+\vec{Q}(qn+r_q)}),
\]
$\vec{Q}(n)$ by $\frac{1}{q}\vec{Q'}(nq+r_q)$ and $\sigma$ by $\rho$. Note that $\frac{1}{q} P(nq+r_q)$ and the components of $\frac{1}{q} \vec{Q'}(nq+r_q)$ are still jointly intersective, by Lemma \ref{lem:basic_intersective} (i). The analog of Claim 2 is
\[
\liminf_{N \to \infty} \frac{1}{N} \sum_{n=1}^N d^*(C \cap (C+P(qn+r_q))) \cdot \nu(1_B \cdot 1_{B+\vec{Q}(qn+r_q)}) >0.
\]
From the modified versions of Claims 1 and 2, we can proceed in the same way as before and construct a KMF mean supported on $S'$.
\end{proof}

The following is a prime version of Theorem \ref{th:improved_bulinski_fish_polynomial}.

\begin{theorem} \label{th:improved_bulinski_fish_polynomial_prime}
Let $P, Q_1, \ldots, Q_d$ be jointly intersective polynomials of the second kind,  $A \subseteq \Z, B \subseteq \Z^d$ with $d^*(A)>0$ and $d_{\Z^d}^*(B) >0$. Let
\[
 S = \{ P(p) : p \textup { prime},\ (Q_1(p), \ldots, Q_d(p)) \in B-B\}.
\]
Then $A-A+S$ contains a Bohr set. Furthermore, if $P$ is not in the $\Q$-vector space spanned by $Q_1, \ldots, Q_d$, then $A-A+S$ contains a subgroup of finite index.
 \end{theorem}

\begin{proof}
The proof is essentially the same as the proof of Theorem \ref{th:improved_bulinski_fish_polynomial} and therefore we just highlight the differences. For each $m$, we fix $r_m$ as in the beginning of the proof of Theorem \ref{th:improved_bulinski_fish_polynomial}, then $(r_m,m)=1$. Without loss of generality, we can assume that $Q_1, \ldots, Q_d$ are linearly independent over $\Q$. In the general case, we can pass to a subset
\[
 S' = \{ P(p) : (Q_1(p), \ldots, Q_d(p)) \in B-B, \, p \textup { prime}, \, p \equiv r_q \pmod{q}\}.
\]
for some suitable $q$.

For $\alpha \in \T$, we consider
\begin{equation} \label{eq:s-alpha-prime}
S_N (\alpha) = \frac{\log N}{N} \sum_{\substack{p=1,\\ p \textup{ prime}}}^N e(P(p) \alpha) \cdot \nu(1_B \cdot 1_{B+\vec{Q}(p)}).
\end{equation}
Similarly to Claim 1 in the proof of \ref{th:improved_bulinski_fish_polynomial}, we can show for all but countably many $\alpha \in \T$, we have $\lim_{N \rightarrow \infty} S_N(\alpha) =0 $. Furthermore, if $P$ is not in the $\Q$-vector space spanned by $Q_1, \ldots, Q_d$, then the exceptional $\alpha$'s are all rational. However, for this, instead of using  Weyl's equidistribution theorem, we use Rhin's classic result \cite{Rhin}, which says that if $f \in \Z[x]$ is a polynomial and at least one of the coefficients of $f-f(0)$ is irrational, then $\displaystyle \lim_{N \rightarrow \infty} \frac{\log N}{N} \sum_{\substack{p=1,\\ p \textup{ prime}}}^N e(f(p)) = 0$. 

The analog of Claim 2 in the proof of \ref{th:improved_bulinski_fish_polynomial} is that
\[
    \liminf_{N \to \infty} \frac{\log N}{N} \sum_{\substack{p=1,\\ p \textup{ prime}}}^N d^*(C \cap (C+P(p))) \cdot \nu(1_B \cdot 1_{B+\vec{Q}(p)}) > 0,
\]
for all $C \subseteq \Z$ satisfying $d^*(C) > 0$. The proof is essentially the same, with Rhin's theorem in place of Weyl's, and using the fact that $(r_m, m)=1$.



Now consider the means $m_N$ given by 
\[
    m_N(f):=\frac{ \sum_{p=1,\, p \textup{ prime}}^N f(P(p)) \cdot \nu(1_B \cdot 1_{B+ \vec{Q}(p)})}{\sum_{p=1, \, p \textup{ prime}}^N \nu(1_B \cdot 1_{B + \vec{Q}(p)})}
\]
and let $m$ be any weak$^*$ cluster point of $(m_N)$. Then it can be shown that $m$ is a $\KMF$ mean supported on $S$. The rest of the proof follows the proof of \cref{th:improved_bulinski_fish_polynomial} verbatim.
\end{proof}

Before addressing the result regarding the set $\{\lfloor n^c \rfloor: n \in \N\}$, we want to introduce a large class of sequences to which $(\lfloor n^c \rfloor)_{n \in \N}$ belongs. Let $B$ denote the set of germs at infinity of smooth real functions on $\R$. Observe that $B$ forms a ring with respect to pointwise addition and multiplication. A \emph{Hardy field} is a subfield of $B$ closed under differentiation. Note that if $f$ is in a Hardy field, the limit $\lim_{x \to \infty} f(x)$ exists in $[- \infty, + \infty]$. Many familiar functions belong to some Hardy field, including functions obtained from $e^x, \log x$ and polynomial functions by composition and arithmetic operations. On the other hand, (non-constant) periodic functions do not belong to any Hardy field. A \emph{Hardy field sequence} is a sequence of the form $(\lfloor f(n) \rfloor)$ where $f: \R \to \R$ belongs to some Hardy field. Besides $( \lfloor n^c \rfloor)$ where $c > 0$, other examples of Hardy field sequences are $(\lfloor n \log n \rfloor)$, $(\lfloor n^2 \sqrt{2} + n \sqrt{3} \rfloor)$, $(\lfloor n^3 + (\log n)^3 \rfloor)$. Two functions $a_1(t), a_2(t)$ in a Hardy field is said to have \emph{different growth rates} if $\lim_{t \to \infty} a_1(t)/a_2(t) = 0$ or $\infty$.

Let $G$ be a nilpotent Lie group and $\Gamma$ be a discrete cocompact subgroup of $G$. The quotient space $X = G/\Gamma$ is called a \emph{nilmanifold}. A nilmanifold has a probability Haar measure $m_X$ which is the image of Haar measure on $G$ under the projection map. The only example of nilmanifold we use in this paper is the torus $\T = \R/\Z$ and in this case the probability Haar measure on $\T$ coincides with the Lebesgue measure.
We say a sequence $(x_n)$ in a nilmanifold $X$ is \emph{equidistributed} in $X$ if for every continuous function $f:X \to \C$,
\[
    \lim_{N \to \infty} \frac{1}{N} \sum_{n=1}^N f(x_n) = \int_X f \ d m_X
\]

\begin{lemma}[{\cite[Theorem 1.3]{Frantzikinakis09}}]
\label[lemma]{lem:frantzikinakis_equidistribution}
Suppose that the functions $a_1(t), \ldots, a_{\ell} (t)$ belong to the same Hardy field, having different growth rates, satisfying 
\begin{equation}\label{eq:need_this_to_equidistribution}
    \lim_{t \to \infty} \frac{a_i(t)}{t^{k_i} \log t} = \lim_{t \to \infty} \frac{t^{k_i + 1}}{a_i(t)} = \infty \text{ for some $k_i \in \N$}.
\end{equation}
If $X_i = G_i/\Gamma_i$ are nilmanifolds, then for every $b_i \in G_i$ and $x_i \in G_i$, the sequence
\[
    (b_1^{\lfloor a_1(n) \rfloor} x_1, \ldots, b_{\ell}^{\lfloor a_{\ell}(n) \rfloor} x_{\ell})_{n \in \N}
\]
is equidistributed in the nilmanifold $\overline{\{b_1^n x_1: n \in \N\}} \times \cdots \times \overline{\{b_{\ell}^n x_{\ell}: n \in \N\}}$.
\end{lemma}

From Lemma \ref{lem:frantzikinakis_equidistribution} we immediately deduce the following.

\begin{lemma}\label[lemma]{lem:corollary_hardy_field}
Let $c_1, c_2 \ldots, c_d > 1$ be distinct non-integers and suppose $\alpha_1, \ldots, \alpha_d \in \T$ are not all equal to 0. Then
\[
    \lim_{N \to \infty} \frac{1}{N} \sum_{n=1}^N e(\lfloor n^{c_1} \rfloor \alpha_1 + \cdots + \lfloor n^{c_d} \rfloor \alpha_d) = 0.
\]
\end{lemma}

\begin{proof}
    The condition \eqref{eq:need_this_to_equidistribution} is clearly satisfied for the functions $a_i(t) = t^{c_i}$, $i=1, \ldots, d$. By Lemma \ref{lem:frantzikinakis_equidistribution}, the sequence $\left( (\lfloor n^{c_1} \rfloor \alpha_1, \ldots, \lfloor n^{c_d} \rfloor \alpha_d) \right)_{n \in \N}$ is equidistributed on $X = X_1 \times \cdots \times X_d$ where $X_i = \overline{\{n \alpha_i: n \in \N\}}$. For $i \in \{1, \ldots, d\}$, let $m_{X_i}$ be the probability Haar measure on $X_i$. We then get
\[
    \lim_{N \to \infty} \frac{1}{N} \sum_{n=1}^N e(\lfloor n^{c_1} \rfloor \alpha_1 + \cdots + \lfloor n^{c_d} \rfloor \alpha_i) = \prod_{i=1}^d \int_{X_i} e(x) \ d m_{X_i}(x).
\]
Since $\alpha_1. \ldots, \alpha_d$ are not all equal to 0, at least one of the integrals on the right hand side is zero and so we are done.
\end{proof}

\begin{theorem}\label{thm:bulinski_fish_hardy_field}
Let $c, c_1, \ldots, c_d > 1$ be non-integers.
Let $A \subseteq \Z$, $B \subseteq \Z^d$ with $d_{\Z}^*(A), d_{\Z^d}^*(B) >0$ and
\[
    S = \{ \lfloor n^c \rfloor : ( \lfloor n^{c_1}  \rfloor, \ldots \lfloor n^{c_d}  \rfloor) \in B-B\}.
\]
Then $A-A+S$ contains a Bohr set. Furthermore, if $c \not \in \{c_1, \ldots, c_d \}$, then $A-A+S = \Z$. 
\end{theorem}

\begin{remark}
For the first conclusion of \cref{thm:bulinski_fish_hardy_field}, we can replace $\lfloor n^c \rfloor, \lfloor n^{c_1} \rfloor, \ldots, \lfloor n^{c_d} \rfloor$ with any $d+1$ sequences $\lfloor a(n) \rfloor, \lfloor a_1(n) \rfloor, \ldots, \lfloor a_d(n) \rfloor$ such that $a(t), a_1(t), \ldots, a_d(t)$ belong to the same Hardy field and satisfy \eqref{eq:need_this_to_equidistribution}. The same is true for the ``furthermore'' part under the additional assumption that $a(t)$ and $a_i(t)$ have different growth rates for all $i \in \{1, \ldots, d\}$.
\end{remark}
\begin{proof}
Again since the proof is similar to that of Theorem \ref{th:improved_bulinski_fish_polynomial}, we only highlight the differences. In fact, due to the nice equidistribution property of sequences along $\lfloor n^c \rfloor$, the proof is simpler in many places.

Let $\vec{F}(n) = ( \lfloor n^{c_1}  \rfloor, \ldots \lfloor n^{c_d} \rfloor)$.
First we assume that $c_1, \ldots, c_d$ are distinct. Consider

\begin{equation} 
S_N (\alpha) = \frac{1}{N} \sum_{n=1}^N e(\lfloor n^c \rfloor \alpha) \cdot \nu(1_B \cdot 1_{ B+ \vec{F}(n)}).
\end{equation}

\noindent \textbf{Claim 1:} For all but countably many $\alpha \in \T$, we have $\lim_{N \rightarrow \infty} S_N(\alpha) =0 $. Furthermore, if $c \not \in \{ c_1, \ldots, c_d\} $, then the only exceptional $\alpha$ is $0$.

\begin{proof}[Proof of Claim 1]
There is a measure $\sigma$ on $\T^d$ such that 
\begin{equation} \label{eq:measure_sigma}
\nu(1_B \cdot 1_{B + \vec{F}(n)}) = \int_{T^d} e \left( \vec{F}(n) \cdot \vec{\beta} \right) d \sigma(\vec{\beta}).    
\end{equation}
Consequently,
\begin{eqnarray*}
S_N (\alpha) &=& \frac{1}{N} \sum_{n=1}^N e(\lfloor n^c \rfloor \alpha) \int_\T e \left( \vec{F}(n) \cdot \vec{\beta} \right) \, d \sigma(\vec{\beta})\\
&=& \int_\T  \left( \frac{1}{N} \sum_{n=1}^N e \left(\lfloor n^c \rfloor \alpha + \vec{F}(n) \cdot \vec{\beta} \right) \right) \, d \sigma(\vec{\beta}).
\end{eqnarray*}
If $c \not \in \{c_1, \ldots, c_d\}$ and $\alpha \neq 0$, then by Lemma \ref{lem:corollary_hardy_field}, for all $\vec{\beta} \in \T^d$,
\[
    \lim_{N \to \infty} \frac{1}{N} \sum_{n=1}^N e(\lfloor n^c \rfloor \alpha + \vec{F}(n) \cdot \vec{\beta} ) = 0.
\]
It follows that $\lim_{N \to \infty} S_N (\alpha) = 0$ in this case.

Suppose $c = c_j$ for some $j \in \{1, \ldots, d\}$. Then again by Lemma \ref{lem:corollary_hardy_field},
\[
    \lim_{N \to \infty} \frac{1}{N} \sum_{n=1}^N e(\lfloor n^c \rfloor \alpha + \vec{F}(n) \cdot \vec{\beta} ) = 0,
\]
unless $\alpha = - \beta_j$ and $\beta_i = 0$ for all $i \neq j$. Therefore, by the dominated convergence theorem, $\lim_{N \rightarrow \infty} S_N(\alpha) = 0$
unless $\alpha = -\beta_j$ for $1 \leq j \leq d$ and for some atom $\vec{\beta}$ of the measure $\sigma$. Our claim follows because the set of such atoms $\vec{\beta}$ is at most countable.
\end{proof}

\noindent \textbf{Claim 2:}
Let $C$ be any subset of $\Z$ with $d^*(C) >0$. We have
\[
    \liminf_{N \to \infty} \frac{1}{N} \sum_{n=1}^N d^*(C \cap (C + \lfloor n^c \rfloor)) \cdot \nu(1_B \cdot 1_{ B+ \vec{F}(n)}) >0.
\]
\begin{proof}[Proof of Claim 2]
As in the proof of Theorem \ref{th:improved_bulinski_fish_polynomial}, there exist measures $\tau$ on $\T$ and $\sigma$ on $\T^d$ such that $\tau(\{0\}) \geq d^*(C), \sigma(\{ \vec{0}\}) \geq d_{\Z^d}^*(B)$ such that
\begin{align*}
     \lim_{N \rightarrow \infty} \frac{1}{N} & \sum_{n=1}^N d^*(C \cap (C + \lfloor n^c \rfloor)) \cdot \nu(1_B \cdot 1_{B+ \vec{F}(n)}) = \\
     & \int_{\T^{d+1}} \lim_{N \to \infty} \frac{1}{N} \sum_{n=1}^N e \left( \lfloor n^c \rfloor \alpha + \vec{F}(n) \cdot \vec{\beta} \right) \ d (\tau \times \sigma) (\alpha, \vec{ \beta}).    
\end{align*}
By Lemma \ref{lem:corollary_hardy_field}, for each $(\alpha, \vec{\beta}) \in \T^{d+1}$, the limit on the right hand is equal to either $0$ or $1$. Since at $(\alpha, \vec{\beta}) = (0, \vec{0})$, the limit is equal to $1$, we infer that the right hand side is at least $\tau(\{0\}) \sigma(\{\vec{0}\}) > 0$, as was to be shown.
\end{proof}

Now consider the means $m_N$ given by 
\[
    m_N(f):=\frac{ \sum_{n=1}^N f(\lfloor n^c \rfloor) \cdot \nu(1_B \cdot 1_{B+\vec{F}(n))}}{\sum_{n=1}^N \nu(1_B \cdot 1_{B+\vec{F}(n)})}
\]
and let $m$ be any weak$^*$ cluster point of $(m_N)$. Then using Claims 1 and 2, it can be shown that $m$ is a $\KMF$ mean supported on $S$. The rest of the proof follows the proof of Theorem \ref{th:improved_bulinski_fish_polynomial} verbatim. Furthermore, if $c \not \in \{ c_1, \ldots, c_d\}$, by Claim 1, the spectrum of $m$ is trivial and so by \cref{th:SpectralSumDifference} $A - A + S = \Z$ if $d^*(A) > 0$.

In the general case where $c_1, \ldots, c_d$ are not necessarily distinct, we may assume that $c_1, \ldots, c_\ell$ are all the distinct values of $\{ c_1, \ldots, c_d \}$. Writing $\vec{F'}(n) = ( \lfloor n^{c_1}  \rfloor, \ldots \lfloor n^{c_\ell} \rfloor)$. Then there is a linear map $\Psi: \T^d \rightarrow \T^\ell$ such that for all $\vec{ \beta} \in \T^d$, we have
\[
\vec{F}(n) \cdot \vec{\beta} = \vec{F'}(n) \cdot \Psi (\vec{\beta}).
\]
Then we can rewrite \eqref{eq:measure_sigma} as
\begin{equation} 
\nu(1_B \cdot 1_{B + \vec{F}(n)}) = \int_{T^d} e \left( \vec{F}(n) \cdot \vec{\beta} \right) d \sigma(\vec{\beta}) = \int_{T^\ell} e \left( \vec{F'}(n) \cdot \vec{\gamma} \right) d \rho(\vec{\gamma}),
\end{equation}
where $\rho = \Psi_*(\sigma)$ is the pushforward of $\sigma$ by $\Psi$, thus reducing the general case to the special case.
\end{proof}



\subsection{Proof of \texorpdfstring{\cref{cor:Specific-intro_main}}{Theorem B}}

We now derive \cref{cor:Specific-intro_main} as a quick corollary of the results in \cref{sec:bulinski_fish}. For the reader's convenience, we first restate \cref{cor:Specific-intro_main}.

\begin{theorem*}
    Let $A \subseteq \Z$ with $d^*(A) > 0$. 
\begin{enumerate}
        \item
        Let $P \in \Z[x]$ be a nonconstant intersective polynomial and $S:=\{P(n):n\in \mathbb N\}$. Then $A - A + S$ contains a finite index subgroup of $\Z$.

       \item\label{item:PolyPrime} Let $P \in \Z[x]$ be a nonconstant intersective polyomial of the second kind and let $S := \{P(p): p \text{ prime}\}$. Then $A - A + S$ contains a finite index subgroup of $\Z$.

        \item 
        If $S = \{\lfloor n^c \rfloor: n \in \N\}$ where $c > 0, c \not \in \Z$, then $A - A + S = \Z$.

        \item 
        Moreover, if $S$ is any of the sets above and $S = S_1 \cup S_2 \cup \cdots \cup S_k$, then one of the $S_i$ is $\DiffBohr$.
    \end{enumerate}
\end{theorem*}

\begin{proof}
(i) Let $P \in \Z[x]$ be a nonconstant intersective polynomial. Then the pair $P, Q_1 = P^2$ are jointly intersective and linearly independent over $\Q$. Therefore, by \cref{th:improved_bulinski_fish_polynomial}, for any $A \subseteq \Z$ with $d^*(A) > 0$,
\[
    A - A + \{P(n): n \in \N\} \supseteq A - A + \{P(n): n \in \N,\ Q_1(n) \in \Z\}
\]
contains a finite index subgroup of $\Z$.

(ii) Let $P \in \Z[x]$ be a nonconstant intersective polynomial of the second kind. 
The pair of polynomials $P$ and $Q_1 = P^2$ are jointly intersective of the second kind and linearly independent over $\Q$. 
By \cref{th:improved_bulinski_fish_polynomial_prime},
\[
    A - A + \{P(p): p \in \P\} \supseteq A - A + \{P(p): p \in \P,\ Q_1(p) \in \Z\}
\]
contains a finite index subgroup of $\Z$. 

(iii) If $0 < c \leq 1$, then $S \supseteq \N$. It is obvious that $A - A + S = \Z$. If $c > 1$, then the conclusion follows directly from \cref{thm:bulinski_fish_hardy_field} with $c_1 = c + 1$ and $B = \Z$.

(iv) From the proofs of Theorems \ref{th:improved_bulinski_fish_polynomial}, \ref{th:improved_bulinski_fish_polynomial_prime}, and \ref{thm:bulinski_fish_hardy_field}, the sets $S$ appearing (i), (ii) and (iii) all support some $\KMF$ means. By \cref{prop:partition_regular_KMF_mean}, if $S = S_1 \cup S_2 \cup \cdots \cup S_k$, one of the $S_i$ supports a $\KMF$ mean and so is $\DiffBohr$ according to \cref{th:SpectralSumDifference_intro}.
\end{proof}


\section{\texorpdfstring{$\DiffBohr$}{(D, B)-expanding}, difference sets, and some counterexamples}

The goal of this section is to prove Theorems \ref{mainthm:difference_set_not_DF} and \ref{thm:A+B_not_pwBohr}. We also provide some results on the relationships between $\DiffBohr$ sets, sets of strong recurrence, and the (un)necessity of annihilating continuous measures in \cref{th:SpectralSumDifference_intro}. 
Before going into the proofs, we need some preliminary definitions and lemmas from \cite{Griesmer_SeparatingBohr}. 

\subsection{Difference sets and \texorpdfstring{$\DiffBohr$}{DB-expanding}}\label{sec:nonDBexample}

For $x \in \R$, let $\lVert x \rVert$ denote the distance to the nearest integer. This norm can also be thought to distance from $x$ to $0$ in $\T = \R/\Z$. For $\epsilon > 0, d \in \N$, and $x = (x_1, \ldots, x_d) \in \T^d$, let
\[
    w_{\epsilon}(x) = |\{j: \lVert x_j \rVert \geq \epsilon\}.
\]
Thus $w_{\epsilon}(x)$ is the number of coordinates of $x$ differing from $0$ by at least $\epsilon$.

For $k < d \in \N$ and $\alpha \in \T^d$, the \emph{Bohr-Hamming ball} with rank $d$ and radius $(k, \epsilon)$ around $0$ determined by $\alpha$ is
\[
    BH(\alpha; k, \epsilon) = \{n \in \Z: w_{\epsilon}(n \alpha) \leq k\}.
\]
Thus $BH(\alpha; k, \epsilon)$ contains all the times $n \in \Z$ such that at most $k$ coordinates of $n \alpha$ differ from $0$ by at least $\epsilon$. If $\{n \alpha: n \in \Z\}$ is dense in $\T^d$ (equivalently, the coordinates of $\alpha$ and $1$ are linearly independent over $\Q$), then $BH(\alpha; k, \epsilon)$ is called a \emph{proper} Bohr-Hamming ball.

For $k \in \N$, we say a set $E \subseteq \Z$ is \emph{$k$-Bohr dense} if $E$ intersects any Bohr neighborhood of rank $k$ of every integer or equivalently for every $n \in \Z$, $E - n$ intersects a Bohr set of rank $k$. We say $E$ is \emph{Bohr dense} if it is $k$-Bohr dense for every $k \in \N$. Note that $E$ is Bohr dense if and only if $E$ is dense in the Bohr compactification $b\Z$ of $\Z$.

\begin{lemma}[see {\cite[Lemma 4.2]{Griesmer_SeparatingBohr}}]
\label[lemma]{lem:bohr-hamming-ball-dense}
For $k < d \in \N, \epsilon > 0$, any proper Bohr-Hamming ball with rank $d$ and radius $(k, \epsilon)$ around $0$ is $k$-Bohr dense. 
\end{lemma}

\begin{lemma}[see {\cite[Lemma 3.5]{Griesmer_SeparatingBohr}}]
\label[lemma]{lem:finite_bohr_dense}
If $S \subseteq \Z$ is $k$-Bohr dense, then for all $M \in \N$ and $\epsilon > 0$, there exists a finite set $S' \subseteq S$ such that for all $m \in \Z$ with $|m| \leq M$, the set $S' - m$ intersects every $(k, \epsilon)$-Bohr set.
\end{lemma}

For $\delta \geq 0$, we say $S \subseteq \Z$ is \emph{$\delta$-nonrecurrent} if there is a set $A \subseteq \Z$ having $d^*(A) > \delta$ and $(A - A) \cap S = \varnothing$. 

\begin{lemma}[see {\cite[Lemma 4.3]{Griesmer_SeparatingBohr}}]
\label[lemma]{lem:bohr-hamming-recurrence}
Let $\delta > 0$ and $k \in \N$. If $S \subseteq \Z$ is finite and $\delta$-nonrecurrent, then there is an $\eta > 0$ and a proper Bohr-Hamming ball $BH$ of radius $(k, \eta)$ such that $S + BH$ is $\delta$-nonrecurrent.
\end{lemma}

\begin{lemma}[see {\cite[Lemma 2.2]{Griesmer_SeparatingBohr}}]
\label[lemma]{lem:finite_nonrecurrent_extrapolate}
Let $0 < \delta < \delta'$ and $S \subseteq \Z$. If every finite subset of $S$ is $\delta'$-nonrecurrent, then $S$ is $\delta$-nonrecurrent.
\end{lemma}

\cref{mainthm:difference_set_not_DF} is an immediate corollary of the following theorem.

\begin{theorem}\label{th:ultimateCounterExample}  
For all $\delta < 1/2$, there are sets $A, S\subseteq \mathbb Z$ such that $d^*(A) > \delta$, $S$ is Bohr-dense, and $(A - A) + (S - S)$ does not contain a Bohr neighborhood of any integer.
\end{theorem}

\begin{proof}



Our construction is similar to that of \cite[Section 4.2]{Griesmer_SeparatingBohr}. 
Fix $\delta'$ such that $\delta < \delta' < 1/2$. We will construct a sequence of finite sets $S_0 \subseteq S_1 \subseteq S_2 \subseteq \ldots$ such that
\begin{enumerate}
    \item $S_k=-S_k$,
    \item for $|m| \leq k$, $S_k - m$ intersects every $(k, 1/k)$-Bohr set,
    \item and $1+S_k+S_k+S_k$ is $\delta'$-nonrecurrent.
\end{enumerate}    
First let $S_0 = \{0\}$. It is easy to verify $S_0$ satisfies the above properties for $k = 0$. Now suppose we already have $S_k$ satisfies the stated properties. Apply \cref{lem:bohr-hamming-recurrence} to find $\eta > 0$ and a proper Bohr-Hamming ball $BH = BH(\alpha; 3(k+1), 3\eta)$ such that $1 + S_k + S_k + S_k + BH$ is $\delta'$-nonrecurrent. Note that 
\[
    BH(\alpha; k+1, \eta) + BH(\alpha; k+1, \eta) + BH(\alpha; k+1, \eta) \subseteq BH(\alpha; 3(k+1), 3\eta).
\]
It follows that
\begin{multline*}
    1 + S_k + S_k + S_k + BH(\alpha; k+1, \eta) + BH(\alpha; k+1, \eta) + BH(\alpha; k+1, \eta) \\
    = 1 + (S_k + BH(\alpha; k+1, \eta)) + (S_k + BH(\alpha; k+1, \eta)) + (S_k + BH(\alpha; k+1, \eta))
\end{multline*}
is $\delta'$-nonrecurrent. 

By \cref{lem:bohr-hamming-ball-dense}, $BH(\alpha; k + 1, \eta)$ is $(k+1)$-Bohr dense and so the same holds for $S_k + BH(\alpha; k + 1, \eta)$. Now by \cref{lem:finite_bohr_dense}, there exists a finite set $S_{k+1} \subseteq S_k + BH(\alpha; k+1, \eta)$ such that for every $|m| \leq k+1$, $S_{k+1} - m$ intersects every $(k+1, 1/(k+1))$-Bohr set. Because both $S_k$ and $BH(\alpha; k, \eta)$ are symmetric and $0 \in BH(\alpha; k, \eta)$, we can choose $S_{k+1}$ to be symmetric and $S_k \subseteq S_{k+1}$.


Let $S=\bigcup_{k=1}^\infty S_k$. From the construction, it is easy to see that $S=-S$, and $S$ is Bohr dense. Every finite subset of $1+S+S+S$ is a subset of $1 + S_k + S_k + S_k$ for some $k$ and so is $\delta'$-nonrecurrent. \cref{lem:finite_nonrecurrent_extrapolate} then implies that $1+S+S+S$ is $\delta$-nonrecurrent. Thus there is an $A\subseteq \mathbb Z$ such that $d^*(A)>\delta$ and $(A - A) \cap (1 + S + S + S) = \varnothing$, or equivalently,
\[
    (A - A + S - S) \cap (-1 - S) = \varnothing.  
\]
Since $-1 - S$ is Bohr dense, we deduce that $A - A + S - S$ does not contain a Bohr neighborhood of any integer. 
\end{proof}



\begin{remark}
Since an infinite difference set is a set of nice recurrence and a van der Corput set, \cref{mainthm:difference_set_not_DF} in particular implies that neither being a set of nice recurrence or van der Corput set implies $\DiffBohr$. In fact, \cref{th:ultimateCounterExample} even says more: there exists $S \subseteq \Z$ which is dense in the Bohr topology such that $S - S$ is not $\DiffBohr$.  
\end{remark}





\subsection{Dense sets in Bohr compactification and proof of \texorpdfstring{\cref{thm:A+B_not_pwBohr}}{Theorem D}}
\label{sec:PWBsumsets}


\begin{lemma}\label{lem:PWBdifference} Let $G$ be a discrete abelian group.  If $C\subseteq G$ is piecewise Bohr, then $C-C$ contains a Bohr set.
\end{lemma}

\begin{proof}     
By \cite[Lemma 12.7]{Griesmer_DiscreteSumsets}, if $C \subseteq G$ is piecewise Bohr then there is a Bohr set $B\subseteq G$ such that $C$ contains a shift of every finite subset of $B$, i.e. for all finite $F\subseteq B$, there is a $t_F\in G$ such that $F+t_F\subseteq C$.  For all such $F$, then $F-F = (F+t_F)-(F+t_F)\subseteq C-C$.  It follows that every finite subset of $B-B$ is a subset of $C-C$, and so $B-B\subseteq C-C$.  
\end{proof}

\begin{proof}[Proof of \cref{thm:A+B_not_pwBohr}]



By \cref{th:ultimateCounterExample}, there exist $S, A \subseteq \Z$ for which $\overline{S} = b\Z$, $d^*(A) > 0$, and $(S + A) - (S + A) = (A - A) + (S - S)$ does not contain a Bohr neighborhood of any integer. By \cref{lem:PWBdifference}, $S + A$ is not piecewise Bohr.
\end{proof}

\subsection{Some examples}

While $\DiffBohr$ sets are sets of recurrence, the next proposition shows that they are not necessarily sets of strong recurrence.  To state it, we introduce a stronger property.  A set $S\subseteq G$ is \emph{$\DiffG$} if for every $A\subseteq G$ with $d^*(A)>0$, $A-A+S=G$.  It is easy to verify that $S$ is $\DiffG$ if and only if for all $g\in G$, $g+S$ is a set of recurrence, if and only if every translate of $S$ is $\DiffBohr$.  

\begin{proposition}\label[proposition]{prop:DB-not-strong-recurrence}
If $G$ is a countably infinite discrete abelian group, then there exists a $\DiffG$ set in $G$ which is not a set of strong recurrence.
\end{proposition}
\begin{proof}
\cite[Corollary 1.4]{GriesmerRecurrenceVsStrong} (see also \cite{Ackelsberg, GriesmerRRPD}) shows that there exists a set $S \subseteq G$ such that for all $g \in G$, $g+S$ is a set of recurrence but not a set of strong recurrence.  In fact, it shows that if every translate of $S$ is a set of recurrence, then there is a set $S'\subseteq S$ such that every translate of $S'$ is a set of recurrence, and no translate of $S'$ is a set of strong recurrence.
\end{proof}

The following proposition shows that there are $\DiffG$ sets, and therefore $\DiffBohr$ sets, which do not support a mean that annihilates continuous measures.  This shows that \cref{th:SpectralSumDifference_intro} fails to fully describe how a set can be $\DiffBohr$.
\begin{proposition}\label[proposition]{prop:diffbohr_not_annihilates_continuous_measure}
If $G$ is a countably infinite discrete abelian group, then there exists a $\DiffBohr$ set $S$ such that none of the means supported on $S$ annihilate continuous measures.
\end{proposition}
\begin{proof}

 By \cite[Proposition 6.2]{Ackelsberg} (see also \cite[Proposition 5.2]{GriesmerRRPD}) there is a continuous Borel probability measure $\sigma$ on $\widehat{G}$ such that for every $\epsilon > 0$, every translate of the set 
\[
    Q_{\epsilon} = \left\{ g \in G: \int_{\widehat{G}} |\chi(g) - 1| \ d \sigma(\chi) < \epsilon\right\}
\]
is a set of recurrence.
Note that $Q_{\epsilon} \subseteq S$ where
\[
    S:=\{g \in G: \Re \hat{\sigma}(g)>1- \epsilon\}.
\]
Therefore every translate of $S$ is a set of recurrence. As shown in the proof of \cref{prop:DB-not-strong-recurrence}, $S$ is $\DiffBohr$. If $m$ is a mean supported on $S$, then $m(|\hat{\sigma}|^2) \geq (1 - \epsilon)^2 > 0$ and so $m$ does not annihilate the continuous measure $\sigma$. 
\end{proof}



\section{Almost Bohr sets and Bohr expansions}

\subsection{Characterizations of \texorpdfstring{$\AlmostBohrBohr$}{(AB,B)-expanding} sets}

Now we prove \cref{mainthm:almost_Bohr_complete}, which we recall here for convenience.

\begin{theorem*}
    Let $G$ be a discrete abelian group and $S \subseteq G$. The following are equivalent:
    \begin{enumerate}
        \item \label{item:ABC1} $S$ is $\AlmostBohrBohr$.

        \item \label{item:ABC2} For every Bohr set $B$, $d^*(S \cap B) > 0$.

        \item \label{item:ABC3} For every almost Bohr set $A$, $S \cap A \neq \varnothing$. 
        

    \end{enumerate} 
\end{theorem*}

\begin{proof} 
(\ref{item:ABC1}) $\Rightarrow$ (\ref{item:ABC2}):  Suppose that there exists a Bohr set $B$ such that $d^*(S \cap B) = 0$. Since $B = -B$, we have $d^*((-S) \cap B) = d^*((S \cap B)) = 0$. Let $A = B \setminus ((- S) \cap B)$, then $A$ is an almost Bohr set and $A \cap (-S) = \varnothing$. Therefore, $0 \not \in A+S$ and $A+S$ does not contain a Bohr set. This shows that $S$ is not $\AlmostBohrBohr$.

(\ref{item:ABC2}) $\Rightarrow$ (\ref{item:ABC1}): Suppose $S$ satisfies (\ref{item:ABC2}). Let $A$ be an almost Bohr set. We write $A = B \setminus E$ where $d^*(E) =0$ and $B = \Bohr(\chi_1, \ldots, \chi_d; \epsilon)$ for some $\chi_1, \ldots, \chi_d \in \widehat{G}$. 

We claim that $B \subseteq A+S$. Indeed, let $g \in B$ be arbitrary. Then there exists $\delta < \epsilon$ such that $| \chi_i(g) - 1 | < \delta$ for every $i=1, \ldots, d$. Then $g - B' \subseteq B$, where $B': = \Bohr(\chi_1, \ldots, \chi_d; \epsilon - \delta)$. By the hypothesis, $d^*(S \cap B') >0$. Since $d^*(g - E) =0$, there exists $h \in (S \cap B') \setminus (g-E)$. Hence $g-h \in (g - B') \setminus E \subseteq A$, and
\[
    g = h + (g-h) \in S + A,
\]
as desired.

\eqref{item:ABC2} $\Leftrightarrow$ \eqref{item:ABC3}: We prove the contrapositive. Suppose there exists an almost Bohr set $A$ such that $S \cap A = \varnothing$. There is a set $E \subseteq G$ such that $d^*(E) = 0$ and $B = A \cup E$ contains a Bohr set. We have
\[
    d^*(S \cap B) \leq d^*(S \cap A) + d^*(S \cap E) = 0.
\]

Conversely, if there exists a Bohr set $B$ with $d^*( S \cap B) =0 $, then $A = B \setminus (S \cap B)$ is an almost Bohr set and $S \cap A =\varnothing$. 
\end{proof}

As a first application, we have the following generalization of Bogolyubov's theorem.
\begin{corollary}\label[corollary]{cor:generalization_bogolyubov}
Let $G$ be a discrete abelian group and $A, B \subseteq G$ of positive upper Banach densities. Then $A - A + B - B$ contains a Bohr set.
\end{corollary}
\begin{proof}
By \cref{thm:folner}, the sets $A - A$ and $B - B$ are almost Bohr sets. Thus, it suffices to show that $B - B$ is $\AlmostBohrBohr$.
There exist Bohr set $B_1$ and a set of zero Banach density $E$ such that $B - B = B_1 \setminus E$. For any Bohr set $B_2$, $B_1 \cap B_2$ is a Bohr set. In particular, $d^*(B_1 \setminus B_2) > 0$ and so
\[
    d^*((B- B) \cap B_2) = d^*((B_1 \setminus E) \cap B_2) \geq d^*((B_1 \cap B_2) \setminus E) > 0.
\]
Since $B_2$ is an arbitrary Bohr set, in view of \cref{mainthm:almost_Bohr_complete}, $B - B$ is $\AlmostBohrBohr$.  \end{proof}

\subsection{Central sets}
The goal of this subsection is to prove \cref{thm:discrete_group_central_sets}.

\begin{lemma}\label{lem:BohrIsIPstar}
Let $G$ be a discrete abelian group. If $B\subseteq G$ a Bohr set, then $B$ is an $IP^*$ set.
\end{lemma}
\begin{proof}
Let $B \subseteq G$ be a Bohr set. To show $B$ is $IP^*$, it suffices to show that for every $IP$ set $A$, we have $B \cap A \neq \varnothing$.

Let $bG$ be the Bohr compactification of $G$ and let $\iota: G \to bG$ be the canonical embedding. There exists neighborhood $U$ of $0$ in $bG$ such that such that $B \supseteq \iota^{-1}(U)$. Let $V$ be an open subset of $U$ such that $V - V \subseteq U$. Since $bG$ is compact and $\iota(G)$ is dense in $bG$, there exist $h_1, \ldots, h_k \in G$ such that 
\[
    bG = \bigcup_{i=1}^k (\iota(h_i) + V).
\]

Let $A \subseteq G$ be an $IP$ set. Thus there exists a sequence $(g_n)_{n \in \N} \subseteq G \setminus \{0\}$ such that $\sum_{n \in F} g_n \in A$ for all finite set $F \subseteq \N$. By pigeonhole principle, there exist $n_1 < n_2$ and $j \in \{1, \ldots, k\}$ such that
\[
    \iota (g_{1} + \cdots + g_{n_1}) \text{ and } \iota(g_1 + \cdots + g_{n_2}) \text{ both belong to } \iota(h_j) + V. 
\]
It follows that
\[
    \iota (g_{n_1 + 1} + \cdots + g_{n_2}) \in V - V \subseteq U.
\]
As a result $g_{n_1 + 1} + \cdots + g_{n_2} \in B \cap A$.
\end{proof}

\begin{lemma}\label[lemma]{lem:ps_phi(A)}
Let $G$ be an abelian group and $\phi: G \to G$ be a homomorphism with $[G: \phi(G)] < \infty$. If $A \subseteq G$ is piecewise syndetic, then $\phi(A)$ is piecewise syndetic.
\end{lemma}
\begin{proof}
Since $A$ is piecewise syndetic, there exists a finite set $\{g_1, \ldots, g_k\}$ such that
\[  
    H = (A - g_1) \cup \cdots \cup (A - g_k)
\]
is thick. This means for any finite set $F \subseteq G$, there exists $g \in G$ such that $F - g \subseteq H$ and so $\phi(F) - \phi(g) \subseteq \phi(H)$.

Every finite subset of $\phi(G)$ has the form $\phi(F)$ for some finite set $F \subseteq G$. It follows that every finite subset of $\phi(G)$ has a translate which is a subset of $\phi(H)$. Since $[G:\phi(G)]$ is finite, $\phi(G)$ is syndetic. As a result, $\phi(H)$ is piecewise syndetic.

Due to the partition regularity of piecewise syndeticity and the fact that 
\[
    \phi(H) = (\phi(A) - \phi(g_1)) \cup \cdots \cup (\phi(A) - \phi(g_k)),
\]
the set $\phi(A) - \phi(g_i)$ is piecewise syndetic for some $i$. Therefore, $\phi(A)$ is piecewise syndetic.
\end{proof}

\begin{proof}[Proof of \cref{thm:discrete_group_central_sets}]
Let $G$ be an abelian group, $\phi: G \to G$ be a homomorphism of finite index, and $C \subseteq G$ a central set. By \cref{lem:ps_phi(A)}, since $C$ is piecewise syndetic, $\phi(C)$ is also piecewise syndetic. In particular, $d^*(\phi(C)) > 0$.
By \cref{thm:folner}, $\phi(C) - \phi(C)$ is an almost Bohr set. In light of \cref{mainthm:almost_Bohr_complete}, it suffices to show that for every Bohr set $B$, $d^*(\phi(C) \cap B) > 0$.

By \cite[Lemma 2.9]{Le-Le-compact}, $\phi^{-1}(B)$ is a Bohr set. By Lemma \ref{lem:BohrIsIPstar}, $\phi^{-1}(B)$ is an $IP^*$ set and so it belongs to every idempotent in $\beta G$. Let $\mathcal{F} \subseteq \mathcal{P}(G)$ be a minimal idempotent ultrafilter that contains $C$. From what we discussed, $\phi^{-1}(B) \in \mathcal{F}$. As a consequence, $C \cap \phi^{-1}(B) \in \mathcal{F}$ and therefore $C \cap \phi^{-1}(B)$ is a central set. It follows that $C \cap \phi^{-1}(B)$ is piecewise syndetic. Note that 
\[
    \phi(C) \cap B \supseteq \phi(C \cap \phi^{-1} B). 
\]
Thus by \cref{lem:ps_phi(A)}, $\phi(C) \cap B$
is piecewise syndetic and so has positive upper Banach density.
\end{proof}

\subsection{Relationships with sets of nice recurrence and van der Corput sets}

\label{sec:abb_vanderCorput_nice_recurrence}

We now give a proof of \cref{thm:ABB_vanderCorput_optimal_recurrence}, which states that any $\AlmostBohrBohr$ set is a van der Corput set and a set of nice recurrence.

\begin{proof}[Proof of \cref{thm:ABB_vanderCorput_optimal_recurrence}]

Let $S$ be an $\AlmostBohrBohr$ set and let $\sigma$ be a positive finite Radon measure on $\widehat{G}$. Let $\epsilon > 0$. By \cref{lem:very_general_way_to_get_nice_vdC}, to prove $S$ is both a van der Corput set and a set of nice recurrence, it suffices to show that there exists $h \in S$ satisfying
\[
    \Re(\hat{\sigma}(h)) > \sigma(\{\chi_0\}) - \epsilon.
\]

Since the proof is similar to the proof of \cref{prop:kmf_imply_vanderCorput}, we only sketch the main ideas and highlight the differences.
Write $\sigma = \sigma_d + \sigma_c$, where $\sigma_d$ and $\sigma_c$ are the discrete and continuous parts of $\sigma$, respectively. As in the proof of \cref{prop:kmf_imply_vanderCorput}, we can show
\[
    \hat{\sigma}_d(0) \geq \sigma (\{ \chi_0 \}).
\]
Let 
\[
    B =  \left\{ g \in G: \Re \left(\hat{\sigma}_d(g)\right) > \sigma(\{\chi_0\}) - \frac{\epsilon}{2}\right\}.
\]
Then $B$ contains the Bohr set 
\[
    \Big\{ g \in G: \sum_{\chi \in \Lambda} \left| \overline{\chi}(g) - 1 \right| \sigma_d( \{ \chi \} )\Big\} < \frac{\epsilon}{4}.
\]
Let $E: = \left\{ g \in G: \left| \hat{\sigma}_c(g) \right| \geq \frac{\epsilon}{2}\right\}$. Since $\sigma_c$ is a continuous measure, 
\cref{lem:meanOfSigmaHat} implies that for every invariant mean $m$ on $\ell^\infty(G)$, we have $m(|\hat{\sigma}_c|^2) = 0$. As $1_E \leq (2/\epsilon) |\hat{\sigma}_c|$, it follows that $m(1_E) = 0$. Because this is true for every invariant mean $m$, we have $d^*(E) = 0$.

Since $S$ is an $\AlmostBohrBohr$ set, by \cref{mainthm:almost_Bohr_complete}, the set $S$ intersects the almost Bohr set $B \setminus E$. Let $h$ be an element in the intersection. We have
\[
   \Re(\hat{\sigma}(h)) = \Re(\hat{\sigma}_d(h)) + \Re( \hat{\sigma}_c(h)) > \sigma(\{\chi_0\}) - \frac{\epsilon}{2} - \frac{\epsilon}{2} = \sigma(\{\chi_0\}) - \epsilon. \qedhere
\]
\end{proof}

\subsection{Relationships with sets of pointwise recurrence in \texorpdfstring{$\Z$}{Z} and proof of \texorpdfstring{\cref{prop:pointwise_recurrence_is_ABB}}{Theorem G}}

A \emph{dynamically central syndetic} set in $\Z$ (or \emph{dcS} set for short) is a set that contains a set of the form
\[
    \{n \in \Z: T^n x \in U\}
\]
where $(X, T)$ is a minimal topological $\Z$-system, $x \in X$ and $U \subseteq X$ is a neighborhood of $X$. A set of pointwise recurrence in $\Z$ is a set that intersects every $dcS$ set. A \emph{dynamically central piecewise syndetic} set (or \emph{dcPS} set) is the intersection of a $dcS$ set and a set of pointwise recurrence. Since $\Z$ itself is a set of pointwise recurrence, it follows from the definitions that a set of pointwise recurrence is a $dcPS$ set. Thus, \cref{prop:pointwise_recurrence_is_ABB} is a corollary of the next proposition.


\begin{proposition}\label[proposition]{prop:dcPS_almost_Bohr_complete}
Every dcPS set in $\Z$ is $\AlmostBohrBohr$. 
\end{proposition}
\begin{proof}
Let $A$ be a $dcPS$ set. We will use the characterization \eqref{item:ABC2-intro} in \cref{mainthm:almost_Bohr_complete} to show that $A$ is a $\AlmostBohrBohr$ set. 

By definition of $dcPS$ sets, $A = A_1 \cap A_2$ where $A_1$ contains 
\[
    R(x, U) = \{n \in \Z: T^n x \in U\}
\]
for a minimal $\Z$-system $(X, T)$, a point $x \in X$, a neighborhood $U$ of $x$, and a set of pointwise recurrence $A_2$. 
Let $B$ be a Bohr set. Then $B$ contains
\[
    R(y, V) = \{n \in \Z: S^n y \in V\},
\]
where $(Y, S)$ is a compact abelian group rotation, $y \in Y$, and $V$ is a neighborhood of $y$. (In the current setting of $G = \Z$, we can take $Y$ to be a finite dimensional torus and $S$ is the translate by an element of $Y$.) By \cite[Theorem 9.11]{Furstenberg81}, the point $(x, y)$ in a uniformly recurrent point in the product system $(X \times Y, T \times S)$. In other words, the closure of 
\[
    \{(T \times S)^n (x, y): n \in \Z\}
\]
is a minimal subsystem of $(X \times Y, T \times S)$. Thus,
\[
    A_1 \cap B \supseteq \{n \in \Z: (T \times S)^n (x, y) \in U \times V\}
\]
is a dynamically central syndetic set. Since $A_2$ is a set of pointwise recurrence, $A \cap B = (A_1 \cap B) \cap A_2$ is a dynamically central piecewise syndetic set and so is piecewise syndetic according to \cite[Theorem F]{glasscock_le_pointwise}. In particular, $d^*(A \cap B) > 0$. Since $B$ is an arbitrary Bohr set, by \cref{mainthm:almost_Bohr_complete} \eqref{item:ABC2-intro}, $A$ is a $\AlmostBohrBohr$ set.
\end{proof}

\subsection{Piecewise syndeticity and proof of \texorpdfstring{\cref{mainthm:counter_for_almost_Bohr_0}}{Theorem I}}
The following concepts are well-defined for arbitrary group $G$. However, as we only use them for $\Z$-actions, we only define them in this case. Given a $\mathbb Z$-system $(X,\mu,T)$, an \emph{eigenvector of $T$} with \emph{eigenvalue} $\lambda\in \mathbb C$ is a $\phi \in L^2(\mu)$ satisfying $\phi\circ T = \lambda \phi$ $\mu$-almost everywhere.  When $(X,\mu,T)$ is ergodic, $f\in L^2(\mu)$, we say that $f$ is \emph{orthogonal to the Kronecker factor} of $(X,\mu,T)$ if $f$ is orthogonal to every eigenvector of $T$ in $L^2(\mu)$.

For $t \in \R$, we write $e(t)$ for $e^{2 \pi i t}$. For the proof of \cref{mainthm:counter_for_almost_Bohr_0}, we need the following version of the Wiener-Wintner theorem.

\begin{theorem}[Uniform Wiener-Wintner Ergodic Theorem (see \cite{Bourgain-double_recurrence} or {\cite[Theorem 2.2]{Assani-wiener_wintner}}]
\label{thm:uniform_wiener_wintner}
Let $(X, \mu, T)$ be an ergodic $\Z$-system, and suppose $F \in L^2(\mu)$ is orthogonal to the Kronecker factor of $(X,\mu,T)$. Then for $\mu$-almost every $x \in X$,
\[
    \lim_{N \to \infty} \sup_{t \in \R} \left| \frac{1}{N} \sum_{n=1}^N f(T^n x) e(nt) \right| = 0.
\]
\end{theorem}


\begin{lemma}\label[lemma]{lem:necessary_for_2_recurrence}
If $S \subseteq \Z$ is a set of $2$-recurrence, then for every $\alpha, \beta \in \R$ and $\epsilon > 0$, there exists $n \in S$ such that $\lVert n \beta \rVert < \epsilon$ and  $\lVert n^2 \alpha \rVert < \epsilon$, where $\lVert \cdot \rVert$ denote the distance to the nearest integer.
\end{lemma}
\begin{proof}
Let $S$ be a set of $2$-recurrence and let $\T$ denote the torus $\R/\Z$. Consider the $3$-torus $X = \T^3$ and the skew product transformation on $X$ defined by $T(x, y, z) = (x + \beta, y + \alpha, z + y + \alpha)$. Let $U_{\epsilon}$ be the $(\epsilon/8)$-ball centered at $(0, 0, 0) \in X$. The system $(X, T)$ is not necessarily minimal (it would be minimal if $\alpha, \beta$ are rationally independent). However, it is a nilsystem and so the orbit closure of every point forms a minimal subsystem. (See, for example, \cite{Bergelson_Host_Kra05, Host_Kra18} for reference about nilsystems.) In particular, the orbit closure of $(0, 0,0)$ under $T$ is a minimal system and $U_{\epsilon}$ contains a neighborhood of $(0, 0, 0)$ in this system. 
Since $S$ is a set of $2$-recurrence, there exists $n \in S$ such that
\begin{equation}\label{eq:intersection_first}
    U_{\epsilon} \cap T^{-n} U_{\epsilon} \cap T^{-2n} U_{\epsilon} \neq \varnothing.
\end{equation}
Note that for all $(x, y, z) \in X$,
\[
    T^n(x, y, z) = (x + n \beta, y + n \alpha, z + n y + \binom{n}{2} \alpha),
\]
and
\[
     T^{2n}(x, y, z) = (x + 2n \beta, y + 2n \alpha, z + 2n y + \binom{2n}{2} \alpha).
\]
Let $(x_1, y_1, z_1)$ be a point in the intersection. We then have
\[
    \lVert x_1 \rVert < \epsilon/8,
\]
\[
    \lVert x_1 + n \beta \rVert < \epsilon/8,
\]
\[
    \lVert z_1 \rVert < \epsilon/8,
\]
\[
    \lVert z_1 + n y_1 + \binom{n}{2} \alpha \rVert < \epsilon/8,
\]
\[
    \lVert z_1 + 2n y_1 + \binom{2n}{2} \alpha \rVert < \epsilon/8.
\]
It follows that
\[
    \lVert n \beta \rVert < \epsilon/4,
\]
and
\[
    \lVert n^2 \alpha \rVert = \lVert z_1 + 2n y_1 + \binom{2n}{2} \alpha - 2(z_1 + n y_1 + \binom{n}{2} \alpha) + z_1 \rVert < \epsilon/2.
\]
\end{proof}

\begin{proof}[Proof of \cref{mainthm:counter_for_almost_Bohr_0}]
Consider the minimal nilsystem $(X, T)$ where $X = \T\times\T$ and $T(x, y) = (x + \alpha, y + 2x + \alpha)$. Let $\mu$ denote Haar measure on $X$, and let $\mu_{\mathbb T}$ denote Haar measure on $\mathbb T$. Let $V \subseteq [1/4, 3/4]$ be a compact set with empty interior having $\mu_{\mathbb T}(V)>0$, and let $U := \T \times V$. Then $U$ is a closed, nowhere dense set with $\mu(U) > 0$. For $z \in X$, define
\[
    S_z = \{n \in \Z: T^n z \in U\}.
\]
We show that there a single $z$ such that
\begin{enumerate}
    \item $A + S_z = \Z$ for every almost Bohr set $A$;
    \item $S_z$ is not piecewise syndetic;
    \item $S_z$ is not a set of $2$-recurrence;
    \item  $S_z$ is not an $IP_0$ set.
\end{enumerate}

First, we show for all $z \in X$, $S_z$ is not piecewise syndetic. For contradiction, assume otherwise. Then there exists $k \in \N$ such that $D = S_z \cup (S_z - 1) \cup \cdots \cup (S_z - k)$ is thick. Since $(X, T)$ is minimal, it follows that
\begin{equation}\label{eq:big_union_everything}
    X = \overline{\{T^n z: n \in D\}} \subseteq U \cup T^{-1} U \cup \cdots \cup T^{-k} U.
\end{equation}
Note that $T$ is a homeomorphism with the continuous inverse $T^{-1}(x, y) = (x - \alpha, y - 2x + \alpha)$. For each $n$, then, $T^{-n} U$ is a closed, nowhere dense subset of $X$, contradicting (\ref{eq:big_union_everything}). 

Now we show there exists $z$ so that $S_z + A = \Z$ for any almost Bohr set $A$.
It is well-known that the eigenvalues of the Koopman operator associated to $T$ are $e(k \alpha)$, $k \in \Z$, and the eigenspace corresponding to $e(k\alpha)$ is spanned by the function $(x, y) \mapsto e(kx)$. Since $1_U = 1_{\T \times V}$, it follows that the function $F = 1_U  - \mu(U)$ is orthogonal to the Kronecker factor of $(X,\mu,T)$.
By \cref{thm:uniform_wiener_wintner}, for $\mu$-almost every $z \in X$, we have 
\begin{equation}\label{eq:app_wiener-wintner}
    \lim_{N \to \infty} \frac{1}{N} \sum_{n=1}^N F(T^n z) e(nt) = 0 \text{ for all } t \in \R.
\end{equation}

Fix $z$ so that \eqref{eq:app_wiener-wintner}.
To show $S_z + A = \Z$ for every almost Bohr set $A$, we need to prove that for any Bohr set $A$ and any $n \in \Z$, $(n - A) \cap S_z \neq \varnothing$; or equivalently, $d^*(S_z\cap B)>0$ for every Bohr neighnorhood $B\subseteq \Z$.

Let $B$ be a Bohr neighborhood in $\Z$. By Proposition \ref{lem:BohrFourierEquivalent}, we may fix a uniformly almost periodic function $g:\mathbb Z\to [0,1]$, supported on $B$, with $\alpha:=\lim_{N\to\infty} \frac{1}{N}\sum_{n=1}^N g(n)>0$. Since $g$ is a uniform limit of linear combinations of functions of the form $n\mapsto e(nt)$, (\ref{eq:app_wiener-wintner}) implies \[
\lim_{N\to\infty}\sum_{n=1}^N (1_U(T^nz)-\mu(U))g(n)=0,
\] so $\lim_{N\to\infty} \frac{1}{N}\sum_{n=1}^N 1_U(T^n z)g(n)=\mu(U)\alpha>0$.   Since $1_B(n)\geq g(n)$, this implies 
\[\lim_{N\to \infty} \frac{1}{N} \sum_{n=1}^N 1_U(T^nz)1_B(n)>0.\] In other words, $\lim_{N\to \infty} \frac{1}{N}\sum_{n=1}^N 1_{S_z\cap B}(n)>0$. Since $\{1,\dots,N\}$ forms a F{\o}lner sequence, this implies $d^*(S_z\cap B)>0$.

Since \eqref{eq:app_wiener-wintner} for $\mu$-almost every $z$, we may also choose $z = (x_1, y_1)$ to satisfy $\lVert y_1 \rVert < 1/8$ and \eqref{eq:app_wiener-wintner} simultaneously. We will show $S = S_z$ is not a set of  $2$-recurrence. For contradiction, assume this is not the case. By \cref{lem:necessary_for_2_recurrence}, there exists $n \in S$ such that
\[
    \lVert 2 n x_1 + n^2 \alpha \rVert < 1/8.
\]
It follows that
\begin{equation}\label{eq:first_contradiction}
    \lVert y_1 + 2 n x_1 + n^2 \alpha \rVert < 1/4.
\end{equation}
However, since $n \in S$, 
\[
    T^n (x_1, y_1) = (x_1 + n \alpha, y_1 + 2 n x_1 + n^2 \alpha) \in U \subseteq \T \times [1/4, 3/4]
\]
and so
\begin{equation}\label{eq:second_contradiction}
    \lVert y_1 + 2 n x_1 + n^2 \alpha \rVert \in [1/4, 3/4].
\end{equation}
Inequality \eqref{eq:first_contradiction} contradicts \eqref{eq:second_contradiction}, so $S$ is not a set of  $2$-recurrence.

It is shown by Furstenberg and Katznelson \cite{Furstenberg_Katznelson85} every $IP_0$ set is a set of  $2$-recurrence. Since $S$ is not a set of  $2$-recurrence, it is not an $IP_0$ set. The set $S$ now satisfies all conditions stated in \cref{mainthm:counter_for_almost_Bohr_0}.
\end{proof}


\subsection{Partition regularity}

\begin{proposition}\label[proposition]{prop:partition_regular_almostBohr}
The family of $\AlmostBohrBohr$ is partition regular; that is if $A = A_1 \cup A_2$ completes almost Bohr sets, then either $A_1$ or $A_2$ is $\AlmostBohrBohr$.
\end{proposition}
\begin{proof}
By \cref{mainthm:almost_Bohr_complete}, the family of sets that complete almost Bohr sets is the dual of the family of almost Bohr sets. Therefore, it suffices to show that the family of almost Bohr sets is a filter. Let $B_1, B_2$ be two Bohr sets and let $E_1, E_2$ have Banach density zero. Then 
\[
    (B_1 \setminus E_1) \cap (B_2 \setminus E_2) = (B_1 \cap B_2) \setminus (E_1 \cup E_2).
\]
Now $B_1 \cap B_2$ is a Bohr set and $d^*(E_1 \cup E_2) = 0$. Therefore, $(B_1 \setminus E_1) \cap (B_2 \setminus E_2)$ is an almost Bohr set.
\end{proof}

\subsection{Relationship with difference sets}

A \emph{difference set} in $\N$ is a set of the from $(A - A) \cap \N$ where $A$ is an infinite subset of $\Z$. As previously demonstrated, $\AlmostBohrBohr$ sets are large (having positive upper Banach density) and combinatorially rich (they are van der Corput sets, sets of nice recurrence, and partition regular). As difference sets possess some of these properties, it is natural to ask whether every $\AlmostBohrBohr$ set contains a difference set. Our next result shows that this is not the case.

This result, when interpreted in the dual form, is equally interesting. A set $A \subseteq \N$ is \emph{$\Delta^*$} if it intersects every difference set. (Due to this definition, we say the families of difference sets and $\Delta^*$-sets are dual -- analogous to the duality between the families of $IP$ sets and $IP^*$ sets.)
A set $A$ is \emph{piecewise Bohr} if $A \supseteq B \cap H$ where $B$ is a Bohr set and $H$ is a thick set. It is shown by Host and Kra \cite{Host_Kra-nilbohr} that every $\Delta^*$-set is a piecewise Bohr set. Because the families of almost Bohr sets and $\AlmostBohrBohr$ sets are dual (\cref{mainthm:almost_Bohr_complete}\eqref{item:ABC3}), our next result implies that not every $\Delta^*$-set is an almost Bohr set (restricted to $\N$). 

\begin{proposition}\label[proposition]{prop:dcps_not_difference}
    There exists an $\AlmostBohrBohr$ set in $\N$ which does not contain a difference set.
\end{proposition}
\begin{proof}
Let $(p_i)$ be an increasing sequence of primes and $(H_i)$ be a sequence of thick sets. It is proved in \cite[Proposition 4.2]{glasscock_le_pointwise} that the set
\[
    A = \bigcup_{i=1}^{\infty} (p_i \N + 1) \cap H_i
\]
is a set of pointwise recurrence. By \cref{prop:dcPS_almost_Bohr_complete}, $A$ is an $\AlmostBohrBohr$ set. 

We now describe a way to construct the sequence of thick sets $(H_i)$ so that $A$ does not contain a difference set. Each $H_i$ is chosen as a union of disjoint intervals of increasing lengths such that
\begin{enumerate}
    \item $H_i \cap H_j = \varnothing$ for $i \neq j$,
    
    \item if $H = \bigcup_{i=1}^{\infty} H_i$, then $H$ can be written as a union of disjoint intervals $\bigcup_{i=1}^{\infty} I_i$ with $\max I_{i} < \min I_{i+1}$ and
\begin{equation}\label{eq:min_max_I}
    \max \{x_1 + \cdots + x_i: x_1 \in I_1, \ldots, x_i \in I_i\} < \min I_{i+1}.
\end{equation}
\end{enumerate}
With these $H_i$, we will show that the corresponding set $A$ does not contain a difference set. For contradiction, assume this is not the case, i.e. there exists an infinite sequence $t_1 < t_2 < \ldots \in \N$ such that $t_j - t_i \in A$ for all $i < j$. For $i \in \N$, let $a_i = t_{i+1} - t_i \geq 1$. It follows that 
\begin{equation}\label{eq:consequecne_difference}
    a_m + a_{m+1} + \cdots + a_n = t_{n+1} - t_m \in A
\end{equation}
for all $1 \leq m \leq n \in \N$.

First, note that $a_1, a_1 + a_2, a_1 + a_2 + a_3, \ldots \in A$. Since $A$ is not syndetic (by \eqref{eq:min_max_I}), the set $\{a_n: n \in \N\}$ is unbounded. Thus, there exist infinitely many $n$ such that $a_n = \max \{a_1, \ldots, a_n\}$. By \eqref{eq:consequecne_difference}, for all $n$, all the numbers
\[
    a_n, a_n + a_{n-1}, \ldots, a_n + a_{n-1} + \cdots + a_1 
\]
belong to $A$ and by \eqref{eq:min_max_I}, they 
belong to the same interval of $H$. Call that interval $I_{i_n}$. From the definition of $A$, all these numbers are elements of $(p_{i_n} \N + 1) \cap I_{i_n}$ for some prime $p_{i_n}$ in the sequence $(p_i)$. It follows that $a_{n-1}, a_{n-2}, \ldots, a_1$ are divisible by $p_{i_n}$. Since the number of prime divisors of $a_1$ is finite, the set
\[
    \{p_{i_n}: a_n = \max \{a_1, \ldots, a_n\}\}
\]
is finite. It follows that there two numbers $m < n \in \N$ so that $p_{i_m} = p_{i_n}$. Therefore, $a_m, a_n \in p_{i_n} \N + 1$. As said before, $a_{n-1}, a_{n-2}, \ldots, a_1$ are all divisible by $p_{i_n}$. In particular, $a_m$ is divisible by $p_{i_n}$; but this contradicts the fact that $a_m \in p_{i_n} \N + 1$.
\end{proof}

\section{Open questions}
In this section, we collect some open questions that arise from our study of $\DiffBohr$ and $\AlmostBohrBohr$ sets.

While \cref{mainthm:almost_Bohr_complete} provides necessary and sufficient conditions for $\AlmostBohrBohr$ sets, \cref{th:SpectralSumDifference_intro} only gives a sufficient condition for $\DiffBohr$ sets. This naturally leads to the following question.

\begin{question}
Find necessary and sufficient conditions for $\DiffBohr$ sets.
\end{question}

On a related note, \cref{th:SpectralSumDifference_intro} states that if a set supports a $\KMF$ mean, then it is $\DiffBohr$. A $\KMF$ mean satisfies two conditions: annihilates continuous measures and massively accumulates at $0$ in $bG$. Between these two conditions, the condition annihilating continuous measures is shown to be unnecessary (see \cref{prop:diffbohr_not_annihilates_continuous_measure}). The next question asks if the other condition is necessary. 

\begin{question}
Does there exist a mean on $\ell^{\infty}(\Z)$ which does not massively accumulate at $0$ in $b\Z$ whose support is $\DiffBohr$?
\end{question}

\cref{prop:partition_regular_KMF_mean} shows that the property of supporting a $\KMF$ mean is partition regular. \cref{prop:partition_regular_almostBohr} proves the same for $\AlmostBohrBohr$ sets. We do not know if the same holds for $\DiffBohr$ sets.
\begin{question}
Is the family of $\DiffBohr$ sets partition regular?    
\end{question}

We show that in \cref{mainthm:difference_set_not_DF} that there is a difference set in $\N$ ($(E - E) \cap \N$ where $E \subseteq \Z$ is infinite) which is not a $\DiffBohr$ set. As a result, van der Corput sets and sets of nice recurrence are not necessarily $\DiffBohr$. On the other hands, $IP$ sets (defined in \cref{sec:central_ip_sets}) are special types of difference sets that possess many rich algebraic and dynamical properties that difference sets do not have. Therefore, in the next question we ask if $IP$ sets are $\DiffBohr$.
\begin{question}
Is every $IP$ set a $\DiffBohr$ set?
\end{question}

In the next two questions, we explore the missing implications in Figure \ref{figure:relations_intersective_prime}.
By \cref{th:SpectralSumDifference_intro} and \cref{prop:partition_regular_almostBohr}, supporting a $\KMF$ mean implies both $\DiffBohr$ and van der Corput. We do not know if $\DiffBohr$ itself already implies van der Corput.

\begin{question}\label{ques:diffbohr_is_vanderCorput}
Is every $\DiffBohr$ set a van der Corput set?
\end{question}

Since every $\DiffBohr$ set is a set of recurrence, any negative answer to \cref{ques:diffbohr_is_vanderCorput} would give an example of a set of recurrence that fails to be a van der Corput set. The first example of this phenomenon was given by Bourgain \cite{Bourgain87}. Recently, Mountakis \cite{Mountakis} refined Bourgain's method to construct a set of strong recurrence that is not van der Corput. Obtaining a negative answer to \cref{ques:diffbohr_is_vanderCorput} would likely demand a construction in a similar spirit.


From their definitions, $\AlmostBohrBohr$ implies $\DiffBohr$. In fact, we believe that $\AlmostBohrBohr$ implies the stronger condition: supporting a $\KMF$ mean.

\begin{conjecture}\label{conj:ab_expanding_implies_kmf}
Every $\AlmostBohrBohr$ set supports a $\KMF$ mean.
\end{conjecture}
A natural way to approach \cref{conj:ab_expanding_implies_kmf} is the following: Let $S \subseteq G$ be $\AlmostBohrBohr$. Then by \cref{mainthm:almost_Bohr_complete}, the set $S \cap B$ has positive upper Banach density for all Bohr set $B$; in particular, $S$ has positive upper Banach density. Let $m$ be an invariant mean on $\ell^{\infty}(G)$ such that $m(1_S) > 0$. Define a mean $\nu$ on $\ell^{\infty}(G)$ by
\[
    \nu(f) = \frac{m(f \cdot 1_S)}{m(1_S)}.
\]

It is obvious that $\nu$ is supported on $S$ and by \cref{prop:RelativeAnnihilates}, the mean $\nu$ annihilates continuous measures. \cref{conj:ab_expanding_implies_kmf} would be proved if we could show that $\nu$ massively accumulates at $0$ in $bG$. Doing so would amount to showing $m(1_{B \cap S}) > 0$ for every Bohr set $B$. Even though this resembles that fact that $d^*(S \cap B) > 0$, what we want to prove is not always true. It is because the mean $m$ may not the capture the positivity of the upper Banach density of $S \cap B$.

Here is a concrete example that demonstrates the failure of this ``natural'' approach. Let $G = \Z$ and let $(I_N)_{N = 1}^{\infty}$ be a sequence of pairwise disjoint intervals in $\N$ such that 
\[
    \frac{|I_N|}{|I_1| + \cdots + |I_{N-1}|} \to \infty \text{ as } N \to \infty.
\]
Let 
\[  
    S = \bigcup_{N = 1}^{\infty} (2 \N \cap I_{2N}) \cup ((2\N + 1) \cap I_{2N + 1}).
\]
Thus $S$ alternatively contains large chunks of even and odd numbers. Since every Bohr set contains syndetically many even numbers, it is not hard to see $S$ is $\AlmostBohrBohr$.
Let $m$ be a mean supported on $\bigcup_{N=1}^{\infty} ((2\N + 1) \cap I_{2N + 1})$. Then $m(1_S) = 1$ but for the Bohr set $B = 2 \Z$, we have $m(1_{B \cap S}) = 0$. This example shows that we need a different approach for \cref{conj:ab_expanding_implies_kmf}.

All the results in our paper are qualitative in nature: we do not provide any bounds on the rank or radius of the Bohr sets appearing in the sumsets. This raises the question of whether some of these results can be strengthened to include quantitative control over these parameters. Here is a concrete example:

\begin{question} \label{q:KMF_quantitative}
Let $G$ be a discrete abelian group and $A, S\subseteq G$. If $S$ supports a $\KMF$ mean and $d^*(A)>0$, is it true that $A - A + S$ contains a Bohr set with rank and radius depending only on $S$ and $d^*(A)$?
\end{question}

For the sets $S$ appearing in parts (i) and (ii) of Theorem \ref{cor:Specific-intro_main}, using a quantitative method different from the one used in this paper, we can confirm Question \ref{q:KMF_quantitative} in the affirmative. For the vast majority of other sets that support $\KMF$ means, however, the answer remains unknown.

\bibliographystyle{abbrv}
\bibliography{refs}

\end{document}